\documentclass[12pt]{amsart}
\newtheorem{theorem}{Theorem}[section]
\newtheorem{lemma}[theorem]{Lemma}

\newtheorem{proposition}[theorem]{Proposition}
\theoremstyle{definition}
\newtheorem{definition}[theorem]{Definition}
\newtheorem{example}[theorem]{Example}

\usepackage{enumerate}
\newtheorem{remark}[theorem]{Remark}
\usepackage[margin=0.75in]{geometry}
\usepackage{calligra}
\usepackage{xcolor, soul}
\sethlcolor{yellow}
\usepackage{mathtools}

\DeclareMathOperator{\sinc}{sinc}

\DeclareMathOperator\Span{span}
\DeclareMathOperator\tr{tr}

\numberwithin{equation}{section}

\def\h{{\mathcal{H}}}
\def\R{{\mathbb{R}^{n}}}
\def\l{{L^2(\mathbb{R}^{2n})}}
\def\L{{L^{2}(\mathbb{T}^{n}\times \mathbb{T}^{n}\times \mathbb{R}^{n})}}
\def\Ll{{L^{2}(\mathbb{T}\times \mathbb{T};\left[\p,\p\right])}}
\def\Z{{\mathbb{Z}^n}}

\def\T{{\mathbb{T}^{n}}}

\def\c{{\mathbb{C}}}

\def\d{{\mathcal{D}}}

\def\t{{T_{(k,l)}^t}}
\def\c{{\xi}}
\def\d{{\xi^{'}}}
\def\e{{\eta}}
\def\p{{\phi}}

\newcommand\numberthis{\addtocounter{equation}{1}\tag{\theequation}}

\newcommand{\be}{\begin{equation}}
	\newcommand{\ee}{\end{equation}}

\begin{document}
	\title[Zak transform associated with the Weyl transform]{Zak transform associated with the Weyl transform and the system of twisted translates on $\mathbb{R}^{2n}$}
	
	\author{Radha Ramakrishnan}
	\address{Department of Mathematics, Indian Institute of Technology, Madras, India}
	\email{radharam@iitm.ac.in}
	\author{Rabeetha Velsamy}
	\address{Department of Mathematics, Indian Institute of Technology, Madras, India}
	\email{rabeethavelsamy@gmail.com}
	
	
	
	
	\begin{abstract}
	We introduce the Zak transform on $L^{2}(\mathbb{R}^{2n})$ associated with the Weyl transform. By making use of this transform, we define a bracket map and prove that the system of twisted translates $\{\t\p : k,l\in \mathbb{Z}^{n}\}$ is a frame sequence iff $0<A\leq \left[\p,\p\right](\xi,\xi^{'})\leq B<\infty,$ for a.e $(\xi,\xi^{'})\in \Omega_{\p},$ where $\Omega_{\p}=\{(\c,\d)\in \T\times\T : \left[\p,\p\right](\xi,\xi^{'})\neq 0\}$. We also prove a similar result for the system $\{\t\p : k,l\in \mathbb{Z}^{n}\}$ to be a Riesz sequence. For a given function belonging to the principal twisted shift-invariant space $V^{t}(\p)$, we find a necessary and sufficient condition for the existence of a canonical biorthogonal function. Further, we obtain a characterization for the system $\{\t\p : k,l\in\mathbb{Z}\}$ to be a Schauder basis for $V^{t}(\p)$ in terms of a Muckenhoupt $\mathcal{A}_{2}$ weight function.
	\end{abstract}
	\maketitle
	\section{Introduction}
	A closed subspace $V\subset L^2(\mathbb{R})$ is called a shift-invariant space if $f\in V\implies T_k f \in V$
	for any $k \in \mathbb{Z}$ where $T_k$ denotes the translation operator $T_kf(y)=f (y-k)$. The spaces appearing in the definition of multiresolution analysis in order to construct wavelets are popular examples of shift-invariant spaces. We refer to \cite{Mallat1989}, \cite{Meyer1987} for further details.\\
	
	Bownik  in \cite{bownik}, characterized shift-invariant spaces on $\R$ in terms of range functions and obtained characterization for a system of translates to form a frame sequence and a Riesz sequence. Later shift-invariant spaces were studied on locally compact abelian groups in \cite{bownikr, cabrelli, kamyabi} and on a non-abelian compact group in \cite{radhas}. Currey et al, in \cite{currey} characterized shift-invariant system in terms of range function for SI/Z type groups. There are several interesting characterization theorems for frames and Riesz sequences in connection with shift-invariant space on various types of Heisenberg group such as polarized Heisenberg group, standard Heisenberg group, A.Weil's abstract Heisenberg group. We refer to \cite{Arati, Mayeli, RHG, das} in this connection.\\
	
	When $V$ is taken to be the principal shift-invariant space $V(\p)=\overline{\Span\{T_{k}\phi : k\in \mathbb{Z}\}}$, for $\p\in L^{2}(\mathbb{R})$, a characterization of the system of translates to form a frame sequence or a Riesz sequence is given in terms of the weight function $w_\phi(\xi)=\sum_{k\in \mathbb{Z}}|\widehat{\phi}(\xi+k)|^2$,\hspace{3mm}$\xi\in [0,1]$. More precisely, the system $\overline{\Span\{T_{k}\phi : k\in \mathbb{Z}\}}$ is a frame sequence with bounds $m,M>0$ iff $m\leq w_\phi(\xi)\leq M$, a.e $\xi\in [0,1]\setminus N$, where $N=\{\xi\in [0,1]:w_\phi(\xi)=0\}$ and it is a Riesz sequence iff $m\leq w_\phi(\xi)\leq M$, a.e $\xi\in [0,1]$. We refer to \cite{CO} for the details. On the otherhand, it is well known that the Gabor system $\{T_{k}M_{l}g : k,l\in \mathbb{Z}\}$ forms a frame for $L^{2}(\mathbb{R})$ iff there exist $A,B>0$ such that $A\leq |Zg(x,\xi)|^{2}\leq B$ a.e $x,\xi\in \mathbb{T}$, where $g\in L^{2}(\mathbb{R})$. Here $Zg$ denotes the Zak transform of $g$ and $M_{l}$ denotes the modulation operator $M_{l}g(y)=e^{2\pi ily}g(y)$. (See \cite{Heilbook}).\\ 
	
	Further, it is well known that the system $\{T_k\phi:k\in \mathbb{Z}\}$ is a Schauder basis for $\overline{\Span\{T_{k}\phi : k\in \mathbb{Z}\}}$ if and only if $w_\phi$ satisfies the Muckenhoupt $\mathcal{A}_2$ condition, where $w_\phi(\xi)=\sum_{k\in \mathbb{Z}}|\widehat{\phi}(\xi+k)|^2$,\hspace{3mm}$\xi\in [0,1]$. We refer \cite{wheeden,Sikic_2007} in this connection.\\
	
	Another interesting problem in the context of shift-invariant spaces is the notion of dual integrability which was introduced in \cite{cyclic} in a more general setting of a locally compact abelian group. Let $G$ be a locally compact abelian group and $d\alpha$ denote Haar measure on the dual group $\widehat{G}.$ A unitary representation $\Pi$ of $G$ on a Hilbert space $\mathcal{H}$ is said to be dual integrable if there exists a function, called the bracket,
	\begin{equation*}
		\left[\cdot,\cdot\right]_{\Pi} : \mathcal{H}\times\mathcal{H}\rightarrow L^{1}(\widehat{G},d\alpha)
	\end{equation*}
	such that 
	\begin{equation}\label{dualdef1}
		\langle \p,\Pi(g)\psi\rangle_{\mathcal{H}}=\int_{\widehat{G}}\left[\p,\psi\right]_{\Pi}(\alpha)\overline{\alpha(g)}\,d\alpha,
	\end{equation}
	for all $\p,\psi\in \mathcal{H}$ and all $g\in G$. Let $\Pi$ be a unitary representation on a separable Hilbert space $\mathcal{H}$. We say $\Pi$ admits a Helson map if there exists a $\sigma-$ finite measure space $(\mathcal{M},\nu)$ and a linear isometry 
	\begin{equation*}
		\mathcal{J} : \mathcal{H}\rightarrow L^{2}((\mathcal{M},\nu), L^{2}(\widehat{G}))
	\end{equation*}
	satisfying 
	\begin{equation}\label{helson1}
		\mathcal{J}[\Pi(g)\p](x)(\alpha)=\alpha(g)\mathcal{J}[\p](x)(\alpha)
	\end{equation}
	for all $g\in G,~\p\in \mathcal{H},$ a.e $x\in \mathcal{M},~\alpha\in \widehat{G}.$
	The following result is well known. 
	\begin{theorem}
		Let G be a locally compact abelian group and $\Pi$ a unitary representation of $G$ on the separable Hilbert space $\mathcal{H}$.
		\begin{itemize}
			\item [(a)] If the representation $\Pi$ admits a Helson map $\mathcal{J}$, then $\Pi$ is dual integrable and a bracket map is provided by \[\left[\p,\psi\right]_{\Pi}(\alpha)=\int_{\mathcal{M}}\mathcal{J}[\phi](x)(\alpha)\overline{\mathcal{J}[\psi](x)(\alpha)}\,dx,\hspace{5mm}\text{for all}~\phi,\psi\in \mathcal{H}.\]
			\item [(b)] If the representation $\Pi$ is dual integrable it admits a Helson map.
		\end{itemize}
	\end{theorem}
	We refer to the recent survey article \cite{tribute} for the above result and several other interesting results including historical developments in connection with bracket map, shift-invariant system, more generally orbits of unitary representations on a locally compact abelian group.\\
	
	In \cite{Iverson}, Iverson studied frames of the form $\{\rho (\xi)f_{i}\}_{\xi\in K,i\in I}$, where $\rho$ is a representation of a non abelian compact group $K$ on a Hilbert space $H_{\rho}$ using operator valued Zak transform. In \cite{Barbieri_ACHA2015, barbieri} Barbieri et al investigated the structure of subspaces of a Hilbert space which are invariant under unitary representations of a discrete group. Here they generalized the concepts of bracket map, fiberization map, dual integrability, non-commutative Zak transform and obtained characterization of frames and Riesz basis in a more general setting using the theory of von Neumann algebra.\\
	
	In this paper, we investigate the above problems for the system of twisted translates $\{\t\p : k,l\in \mathbb{Z}^{n}\}$ on $\mathbb{R}^{2n}$. In \cite{saswata}, Radha and Saswata introduced twisted shift-invariant spaces $V^{t}(\p)$ in $L^{2}(\mathbb{R}^{2n})$ and studied the problem of characterizing the system of twisted translates to form a frame sequence or a Riesz sequence. These characterizations were obtained using certain ``Condition C''. More precisely it was shown that if $\p\in L^{2}(\mathbb{R}^{2n})$ such that $\p$ satisfies ``Condition C'', then $\{\t\p : k,l\in \mathbb{Z}^{n}\}$ is a frame sequence with frame bounds $A,B>0$ iff $A\leq \omega_{\p}(\xi)\leq B$ a.e $\xi\in \Omega_{\p}$, where $\Omega_{\p}=\{\xi\in \T : \omega_{\p}\neq 0\}$ and $\omega_{\p}(\xi)=\sum_{m\in \mathbb{Z}^{n}}\int_{\mathbb{R}^{n}}|K_{\p}(\xi+m,\eta)|^{2}\,d\eta,~\xi\in \mathbb{T}^{n}$. Here $K_{\p}$ denotes the kernel of the Weyl transform of $\p$. A similar result was shown for a Riesz sequence. Recently in \cite{Twisted}, twisted $B-$ splines were introduced and various fundamental properties were studied. It is to be noted that except the first order twisted $B-$ spline, all other splines do not satisfy ``Condition C''. In the paper \cite{Twisted}, for the second order twisted $B-$ spline, it was shown that the corresponding system of twisted translates forms a Riesz sequence by applying numerical computations. These kind of hard computations could not be extended to higher order twisted $B-$ splines. Hence, in this paper, we attempt to obtain characterizations for the system of twisted translates $\{\t\p : k,l\in \mathbb{Z}^{n}\}$ to form a frame sequence and a Riesz sequence without the assumption ``Condition C'' on the generator $\p$.\\
	
	First we introduce the Zak transform associated with the Weyl transform using its kernel. We call this transform, the Weyl-Zak transform. By making use of this transform, we define a bracket map and prove that the system of twisted translates $\{\t\p : k,l\in \mathbb{Z}^{n}\}$ is a frame sequence iff $0<A\leq \left[\p,\p\right](\xi,\xi^{'})\leq B<\infty,$ for a.e $(\xi,\xi^{'})\in \Omega_{\p},$ where $\Omega_{\p}=\{(\c,\d)\in \T\times\T : \left[\p,\p\right](\xi,\xi^{'})\neq 0\}$. We also prove a similar result for the system $\{\t\p : k,l\in \mathbb{Z}^{n}\}$ to be a Riesz sequence. We illustrate this result with an example. We find a necessary and sufficient condition for the existence of the canonical biorthogonal function $\tilde{\p}$ for a given $\p\in V^{t}(\p)$, where $V^{t}(\p)$ denotes the principal twisted shift-invariant space. We also show that if $\p\in L^{2}(\mathbb{R}^{2n})$ and the system $\{\t\p : k,l\in \mathbb{Z}^{n}\}$ forms a Riesz sequence, then there exists a $\tilde{\p}\in V^{t}(\p)$ such that $\{\t\p : k,l\in \mathbb{Z}^{n}\}$ is an orthonormal system for $V^{t}(\p)$. In this case $V^{t}(\tilde{\p})$ coincides with $V^{t}(\p)$. Further, we obtain a characterization for the system $\{\t\p : k,l\in \mathbb{Z}\}$ to be a Schauder basis for $V^{t}(\p)$ in terms of a Muckenhoupt $\mathcal{A}_{2}$ weight function.\\
	
	In the final part of this paper, we look at the notion of dual integrability in connection with twisted translation. First we observe that the function $\Pi : \mathbb{Z}^{n}\times\mathbb{Z}^{n}\rightarrow \mathcal{U}(L^{2}(\mathbb{R}^{2n}))$ defined by $\Pi(k,l)=T^{t}_{(k,l)}$, $k,l\in \mathbb{Z}^{n}$ is only a projective representation. Hence we take the group $H=2\mathbb{Z}^{n}\times\mathbb{Z}^{n}$ and define the map $\Pi_{H} : 2\mathbb{Z}^{n}\times\mathbb{Z}^{n}\rightarrow \mathcal{U}(L^{2}(\mathbb{R}^{2n}))$ by $\Pi_{H}(2k,l)=T^{t}_{(2k,l)}$, $k,l\in \mathbb{Z}^{n}$, which turns out to be a unitary representation of $H$ on $L^{2}(\mathbb{R}^{2n})$. When we try to investigate whether the map $\Pi_{H}$ is dual integrable, we notice that we get an integrable form similar to that of \eqref{dualdef1} but the presence of an additional term in the integrand (See \eqref{bracketpi}). In a similar way, when we want to examine whether $\Pi_{H}$ admits a Helson map, we notice an additional term in comparison with \eqref{helson1} (See \eqref{helsonmap}). We believe that these two observations pave a way for further investigations in future.
	\section{Notation and Background}
	Let $\h\neq 0$ be a separable Hilbert space.
	\begin{definition}
		A sequence $\{f_{k} : k\in\mathbb{Z}\}$ in $\h$ is called a frame for $\h$ if there exist two constants $A,B>0$ such that
		\begin{align*}
			A\|f\|^{2}\leq \sum_{k\in \mathbb{Z}}|\langle f,f_{k}\rangle|^{2} \leq B\|f\|^{2},\hspace{10mm}\forall\ f\in \h.\numberthis \label{frm}
		\end{align*}
	\end{definition}
	If $\{f_{k} : k\in \mathbb{Z}\}$ is a frame for $\overline{\Span\{f_{k} : k\in \mathbb{Z}\}}$, then it is called a frame sequence.
	\begin{definition}
		A sequence $\{f_{k} : k\in \mathbb{Z}\}$ in $\h$ is called a Riesz basis for $\h$ if there exists a bounded invertible operator $U$ on $\h$ and an orthonormal basis $\{e_{k} : k\in \mathbb{Z}\}$ of $\h$ such that $U(e_{k})=f_{k},\ \forall\ k\in\mathbb{Z}$. If $\{f_{k} : k\in \mathbb{Z}\}$ is a Riesz basis for $\overline{\Span\{f_{k} : k\in \mathbb{Z}\}}$ then it is called a Riesz sequence.
	\end{definition}
	\begin{theorem}
		Let $\{f_{k} : k\in \mathbb{Z}\}$ in $\h$ be a sequence in $\mathcal{H}$. Then the following conditions are equivalent.
		\begin{enumerate}[(i)]
			\item $\{f_{k}\}$ is a Riesz basis for $\h$.
			\item $\{f_{k}\}$ is complete in $\h$, and there exist constants $A,B>0$ such that $$A\sum_{k}|c_{k}|^{2}\leq \big\|\sum_{k} c_{k}f_{k}\big\|^{2}\leq B\sum_{k} |c_{k}|^{2},$$ for every finite scalar sequence $\{c_{k}\}$.
		\end{enumerate}
	\end{theorem}
For further study of frames and Riesz bases we refer to \cite{CN}.
\begin{definition}\cite{biorthogonal}
	Let $\mathcal{B}=\{\psi_{k} : k\in \mathbb{Z}\}$ be a countable collection of functions in $\mathcal{H}$ such that $\overline{\Span\{\psi_{k} : k\in \mathbb{Z}\}}=\mathcal{H}.$ A collection $\tilde{\mathcal{B}}=\{\tilde{\psi_{k}} : k\in \mathbb{Z}\}$ is biorthogonal to $\mathcal{B}$ if \[\langle \psi_{j}, \tilde{\psi_{k}}\rangle=\delta_{jk}\hspace{5mm}\text{for all}\hspace{5mm}j,k\in \mathbb{Z}.\]
\end{definition}
\begin{definition}\cite{biorthogonal}
	Let $\psi\in \mathcal{H}$ and $\mathcal{B}=\{T_{k}\psi : k\in \mathbb{Z}\},$ where $T_{k}\psi$ is the integer translate of $\psi.$ If there exists $\tilde{\psi}\in \overline{\Span(\mathcal{B})}$ with the property $\langle T_{k}\psi,\tilde{\psi}\rangle=\delta_{k0},$ then the system $\{T_{j}\tilde{\psi} : j\in \mathbb{Z}\}$ is biorthogonal to $\mathcal{B}$. The function $\tilde{\psi}$, will be called the canonical biorthogonal function to $\psi.$
\end{definition}

\begin{definition}
	An ordered collection $\{f_{n}\}_{n\in \mathbb{N}}$ in a Hilbert space $\mathcal{H}$ is a Schauder basis for $\mathcal{H}$ if for each $f\in \mathcal{H}$ there exist unique scalars $c_{n}(f)$ such that 
	\begin{equation}\label{schauder basis}
		f=\sum\limits_{n\in\mathbb{N}}c_{n}(f)f_{n},
	\end{equation}
	where the series converges in the norm of $\mathcal{H}.$
\end{definition}
Given a sequence $\{f_{n}\}_{n\in \mathbb{N}}$ that has a biorthogonal sequence $\{\tilde{f_{n}}\}_{n\in \mathbb{N}}$, we define the partial sum operators $S_{N}:\mathcal{H}\rightarrow\mathcal{H}$ by \[S_{N}(f):=\sum\limits_{n=1}^{N}\langle f, \tilde{f_{n}}\rangle f_{n}.\] 
\begin{theorem}\label{schauderchara}
	Given a collection $\{f_{n}\}_{n\in\mathbb{N}}$ in a Hilbert space $\mathcal{H},$ the following statements are equivalent.
	\begin{itemize}
		\item [(i)] $\{f_{n}\}_{n\in \mathbb{N}}$ is a Schauder basis.
		\item [(ii)] There exists a biorthogonal sequence $\{\tilde{f_{n}}\}_{n\in \mathbb{N}}$ such that the partial sum operators $S_{N}$ converge in the strong operator topology to the identity map, that is, 
		\[f=\sum\limits_{n\in\mathbb{N}}\langle f,\tilde{f_{n}}\rangle f_{n}\hspace{5mm}~\forall~f\in \mathcal{H}.\]
		\item [(iii)] There exists a biorthogonal sequence $\{\tilde{f_{n}}\}_{n\in \mathbb{N}}$ such that the partial sum operators are uniformly bounded in operator norm, that is, $\sup\limits_{N}\|S_{N}\|<\infty.$
	\end{itemize}
\end{theorem}
\begin{definition}\cite{Heil}
	A non-negative function $w\in L^{1}(\mathbb{T}\times \mathbb{T})$ is in Muckenhoupt $\mathcal{A}_{2}(\mathbb{T}\times \mathbb{T})$ weight if there exists $C>0$ such that for all intervals $I,J\subset \mathbb{R},$ we have \[\bigg(\frac{1}{|I||J|}\int_{J}\int_{I}w(x,y)\,dxdy\bigg)\bigg(\frac{1}{|I||J|}\int_{J}\int_{I}\frac{1}{w(x,y)}\,dxdy\bigg)\leq C.\] 
\end{definition}
The above definition of product $\mathcal{A}_{2}$ weights is equivalent to requiring that $w$ satisfies a uniform $\mathcal{A}_{2}$ condition in each variable. More precisely, $w\in \mathcal{A}_{2}(\mathbb{T}\times \mathbb{T})$ if and only if for almost every $x\in \mathbb{T}$ and for each interval $J\subset \mathbb{R},$ 
\[\bigg(\frac{1}{|J|}\int_{J}w(x,y)\,dy\bigg)\bigg(\frac{1}{|J|}\int_{J}\frac{1}{w(x,y)}\,dy\bigg)\leq C\]
and likewise for the $y$ variable.\\

We now provide the necessary background for the study of Heisenberg group.\\

The Heisenberg group $\mathbb{H}^{n}$ is a nilpotent Lie group whose underlying manifold is $\mathbb{R}^{n}\times\mathbb{R}^{n}\times\mathbb{R}$ with the group operation defined by $(x,y,t)(u,v,s)=(x+u,y+v,t+s+\frac{1}{2}(u\cdot y-v\cdot x))$ and the Haar measure is the Lebesgue measure $dx\,dy\,dt$ on $\mathbb{R}^{n}\times\mathbb{R}^{n}\times\mathbb{R}$. Using the Schr\"{o}dinger representation $\pi_{\lambda},$ $\lambda\in\mathbb{R}^{\ast},$ given by
\begin{equation*}
	\pi_{\lambda}(x,y,t)\p(\xi)=e^{2\pi i\lambda t}e^{2\pi i\lambda(x\cdot\xi+\frac{1}{2}x\cdot y)}\p(\xi+y),~\p\in L^{2}(\mathbb{R}^{n}),
\end{equation*}
we define the group Fourier transform of $f\in L^{1}(\mathbb{H}^{n})$ as
\begin{equation*}
	\widehat{f}(\lambda)=\int_{\mathbb{H}^{n}}f(z,t)\pi_{\lambda}(z,t)\,dzdt,~\text{where}~\lambda\in\mathbb{R}^{\ast},
\end{equation*}
which is a bounded operator on $L^{2}(\mathbb{R}^{n})$. In otherwords, for $\p\in L^{2}(\mathbb{R}^{n})$, we have
\begin{equation*}
	\widehat{f}(\lambda)\p=\int_{\mathbb{H}^{n}}f(z,t)\pi_{\lambda}(z,t)\p\,dzdt,
\end{equation*}
where the integral is a Bochner integral taking values in $L^{2}(\mathbb{R}^{n}).$ If $f$ is also in $L^{2}(\mathbb{H}^{n})$, then $\widehat{f}(\lambda)$ is a Hilbert-Schmidt operator. Define
\begin{equation*}
	f^{\lambda}(z)=\int_{\mathbb{R}}e^{2\pi i\lambda t}f(z,t)\,dt,
\end{equation*}
which is the inverse Fourier transform of $f$ in the $t$ variable. Then we can write
\begin{equation*}
	\widehat{f}(\lambda)=\int_{\mathbb{C}^{n}}f^{\lambda}(z)\pi_{\lambda}(z,0)dz.
\end{equation*}
Let $g\in L^{1}(\mathbb{C}^{n})$. Define
\begin{equation*}
	W_{\lambda}(g)=\int_{\mathbb{C}^{n}}g(z)\pi_{\lambda}(z,0)\,dz.
\end{equation*}
When $\lambda=1$, it is called the Weyl transform of $g$, denoted by $W(g)$. This can be explicitly written as
\begin{equation*}
	W(g)\p(\xi)=\int_{\mathbb{R}^{2n}}g(x,y)e^{2\pi i(x\cdot\xi+\frac{1}{2}x\cdot y)}\p(\xi+y)\,dxdy,~\p\in L^{2}(\mathbb{R}^{n}).
\end{equation*}
The Weyl transform is an integral operator with kernel $K_{g}(\xi,\eta)=\int_{\R}g(x,\eta-\xi)e^{\pi ix\cdot(\xi+\eta)}dx$. If $g\in L^{1}\cap L^{2}(\mathbb{C}^{n})$, then $K_{g}\in L^{2}(\mathbb{R}^{2n})$, which implies that $W(g)$ is a Hilbert-Schmidt operator whose norm is given by $\|W(g)\|^{2}_{\mathcal{B}_{2}}=\|K_{g}\|^{2}_{L^{2}(\mathbb{R}^{2n})}$, where $\mathcal{B}_{2}$ is the Hilbert space of Hilbert-Schmidt operators on $L^{2}(\mathbb{R}^{n})$ with inner product $(T,S)=\tr(TS^{\ast})$. The Plancherel formula and the inversion formula for the Weyl transform are given by $\|W(g)\|^{2}_{\mathcal{B}_{2}}=\|g\|^{2}_{L^{2}(\mathbb{C}^{n})}$ and $g(w)=\tr(\pi(w)^{\ast}W(g))$, $w\in \mathbb{C}^{n}$, respectively. For a detailed study of analysis on the Heisenberg group we refer to \cite{follandphase, thangavelu}.
\begin{definition}\cite{saswata}
	Let $\p \in \l$. For $(k,l)\in \mathbb{Z}^{2n},$ the twisted translation $\t\p$ of $\p$, is defined by
	\begin{align*}
		\t\p(x,y)=e^{\pi i(x\cdot l-y\cdot k)}\p(x-k,y-l),\hspace{5mm}(x,y)\in \mathbb{R}^{2n}.
	\end{align*}
\end{definition}
Using the definition of twisted translation, we have
\begin{align*}
	T_{(k_1,l_1)}^tT_{(k_2,l_2)}^t=e^{-\pi i(k_1\cdot l_2-l_1\cdot k_2)}T_{(k_1+k_2,l_1+l_2)}^t,\ \ \ \ \ \forall\ (k_1,l_1),(k_2,l_2)\in\mathbb{Z}^{2n}.\numberthis  \label{comptsttrns}
\end{align*}
\begin{lemma}\label{twistker}\cite{saswata}
	Let $\p \in \l$. Then the kernel of the Weyl transform of $\t\p$ satisfies
	\begin{align*}
		K_{\t\p}(\xi,\eta)=e^{\pi i(2\xi+l)\cdot k}K_{\p}(\xi+l,\eta).
	\end{align*}
\end{lemma}

\section{Weyl-Zak transform and the bracket map}
We define the Zak transform $Z_{W}$ associated with the Weyl transform using its kernel as follows.
The map $Z_{W} : \l \rightarrow \L$ is defined by \[Z_{W} \p(\xi,\xi^{'},\eta)=\sum\limits_{m\in \Z}K_{\p}(m+\c,\e)e^{-2\pi im\cdot\d}.\] In fact, for $\p\in \l$, we have
\begin{align*}
	\|Z_{W} \p\|^{2}_{\L}&=\int_{\mathbb{T}^{n}}\int_{\mathbb{T}^{n}}\int_{\mathbb{R}^{n}}|Z_{W} \p(\c,\d,\e)|^{2}\,d\e d\d d\c\\
	&=\int_{\mathbb{T}^{n}}\int_{\mathbb{R}^{n}}\int_{\mathbb{T}^{n}}\bigg|\sum\limits_{m\in \Z}K_{\p}(\c+m,\e)e^{-2\pi im\cdot \d}\bigg|^{2}\,d\d d\e d\c\\
	&=\int_{\mathbb{T}^{n}}\int_{\mathbb{R}^{n}}\int_{\T}\sum\limits_{m\in \Z}K_{\p}(\c+m,\e)e^{-2\pi im\cdot\d}\overline{\sum\limits_{m^{'}\in \Z}K_{\p}(\c+m^{'},\e)e^{-2\pi im^{'}\cdot\d}}\,d\d d\e d\c\\
	&=\int_{\mathbb{T}^{n}}\int_{\mathbb{R}^{n}}\sum\limits_{m,m^{'}\in \Z}K_{\p}(\c+m,\e)\overline{K_{\p}(\c+m^{'},\e)}\int_{\T}e^{-2\pi i(m-m^{'})\cdot \d}\,d\d d\e d\c\\
	&=\int_{\mathbb{T}^{n}}\int_{\mathbb{R}^{n}}\sum\limits_{m\in \Z}K_{\p}(\c+m,\e)\overline{K_{\p}(\c+m,\e)}\,d\e d\c\\
	&=\int_{\mathbb{T}^{n}}\int_{\mathbb{R}^{n}}\sum\limits_{m\in \Z}|K_{\p}(\c+m,\e)|^{2}\,d\e d\c\\
	&=\int_{\mathbb{R}^{n}}\int_{\mathbb{R}^{n}}|K_{\p}(\c,\e)|^{2}\,d\e d\c\\
	&=\|K_{\p}\|^{2}_{\l}\\
	&=\|\p\|^{2}_{\l},
\end{align*}
by using Fubini's theorem. Let $F\in\L.$ For a.e $\c\in\T,$ $\e\in \mathbb{R}^{n},$ we have $F(\c,\cdot,\e)\in L^{2}(\T).$ Then the Fourier series expansion of $F(\c,\cdot,\e)$ is given by \[F(\c,\d,\e)=\sum\limits_{m\in\Z}a^{\c,\e}_{m}e^{2\pi im\cdot \d},~\text{for some}~\{a^{\c,\e}_{m}\}_{m\in\Z}\in\ell^{2}(\Z).\] For each $\xi\in \mathbb{R}^{n},$ there exist unique elements $m\in \Z$ and $\gamma\in \T$ such that $\xi=\gamma+m.$ Define a function $K$ on $\mathbb{R}^{2n}$ by $K(\xi,\eta)=a^{\gamma,\eta}_{m}.$ Now, 
\begin{align*}
	\int_{\mathbb{R}^{n}}\int_{\mathbb{R}^{n}}|K(\c,\e)|^{2}\,d\c\,d\e&=\int_{\mathbb{R}^{n}}\int_{\mathbb{T}^{n}}\sum\limits_{m\in \Z}|K(\gamma+m,\e)|^{2}\,d\gamma\,d\e\\
	&=\int_{\mathbb{R}^{n}}\int_{\mathbb{T}^{n}}\sum\limits_{m\in \Z}|a^{\gamma,\e}_{m}|^{2}\,d\gamma\,d\e\\
	&=\int_{\mathbb{R}^{n}}\int_{\mathbb{T}^{n}}\int_{\mathbb{T}^{n}}|F(\gamma,\d,\e)|^{2}\,d\d\,d\gamma\,d\e<\infty,
\end{align*}
which follows from the Plancherel theorem for the Fourier series. By surjectivity of the Weyl transform there exists $f\in L^{2}(\mathbb{R}^{2n})$ such that $W(f)$ is a Hilbert Schmidt operator with the kernel $K\in L^{2}(\mathbb{R}^{2n}).$ Therefore $Z_{W}$ is an isometric isomorphism between $L^{2}(\mathbb{R}^{2n})$ and $\L.$

\vspace{3mm} 

\hspace{3mm}From, now onwards we call $Z_{W}$, the Weyl-Zak transform.
\begin{example}
	Let $\phi(x,y)=\chi_{[0,1)}(x)\chi_{[0,1)}(y).$ Then,
	the kernel of the Weyl transform of $\phi$ is given by \[K_{\phi}(\c,\e)=e^{\frac{\pi i}{2}(\c+\e)}\sinc\bigg(\frac{\c+\e}{2}\bigg)\chi_{[0,1)}(\e-\c)\]
	(See \cite{saswata}). Then we have 
	\begin{align*}
	Z_{W}\phi(\c,\d,\e)&=\sum\limits_{m\in \Z}K_{\p}(m+\c,\e)e^{-2\pi im\cdot\d}\\
	&=\sum\limits_{m\in \Z}e^{\frac{\pi i}{2}(\c+m+\e)}\sinc\bigg(\frac{\c+m+\e}{2}\bigg)\chi_{[0,1)}(\e-\c-m)e^{-2\pi im\cdot\d}\\
	&=\sum\limits_{m\in \Z}e^{\frac{\pi i}{2}(\c+m+\e)}\sinc\bigg(\frac{\c+m+\e}{2}\bigg)\chi_{[m,m+1)}(\e-\c)e^{-2\pi im\cdot\d}\\
	&=e^{\frac{\pi i}{2}(\c+\lfloor\e-\c\rfloor+\e)}\sinc\bigg(\frac{\c+\lfloor\e-\c\rfloor+\e}{2}\bigg)e^{-2\pi i\lfloor\e-\c\rfloor\cdot\d}.
	\end{align*}
	Hence $Z_{W}\phi(\c,\d,\e)=e^{\frac{\pi i}{2}(\c+\lfloor\e-\c\rfloor+\e)}\sinc\bigg(\frac{\c+\lfloor\e-\c\rfloor+\e}{2}\bigg)e^{-2\pi i\lfloor\e-\c\rfloor\cdot\d}.$
\end{example}
The following proposition gives the image of twisted translates of $\p$ under the Weyl-Zak transform.
\begin{proposition}\label{zaktwistprop}
	Let $\p\in \l$. The Weyl-Zak transform of twisted translates of $\p$ is given by 
	\begin{equation}\label{zaktwist}
		Z_{W} \t \p(\c,\d,\e)=e^{2\pi i(k\cdot \c+l\cdot \d)}e^{\pi ik\cdot l}Z_{W} \p(\c,\d,\e), \hspace{5mm}k,l\in \Z,~\c,\d\in \T,~\e\in \mathbb{R}^{n}.
	\end{equation}
\end{proposition}
\begin{proof}
	By making use of Lemma \refeq{twistker}, we have
	\begin{align*}
		Z_{W}\t\p(\c,\d,\e)&=\sum\limits_{m\in \Z}K_{\t\p}(\c+m,\e)e^{-2\pi im\cdot \d}\\
		&=\sum\limits_{m\in \Z}e^{\pi i(2(\c+m)+l)\cdot k}K_{\p}(\c+m+l,\e)e^{-2\pi im\cdot \d}\\
		&=e^{2\pi ik\cdot \c}e^{\pi ik\cdot l}\sum\limits_{m\in \Z}K_{\p}(\c+m+l,\e)e^{-2\pi im\cdot \d}\\
		&=e^{2\pi ik\cdot \c}e^{\pi ik\cdot l}\sum\limits_{m\in \Z}K_{\p}(\c+m,\e)e^{-2\pi i(m-l)\cdot \d}\\
		&=e^{2\pi i(k\cdot \c+l\cdot \d)}e^{\pi ik\cdot l}\sum\limits_{m\in \Z}K_{\p}(\c+m,\e)e^{-2\pi im\cdot \d}\\
		&=e^{2\pi i(k\cdot \c+l\cdot \d)}e^{\pi ik\cdot l}Z_{W}\p(\c,\d,\e).
	\end{align*}
	This completes the proof.
\end{proof} 

\begin{remark}\label{plancherel theorem}
	In this paper, in most of the places we take the orthonormal basis $\{e^{2\pi i(k\cdot\xi+l\cdot \d)}e^{\pi ik\cdot l} : k,l\in \mathbb{Z}^{n}\}$ instead of the classical Fourier series in $L^{2}(\T\times \T)$. This leads to the Plancherel theorem \begin{equation*}
		\int_{\mathbb{T}^{n}}\int_{\mathbb{T}^{n}}|f(x,y)|^{2}\,dxdy=\sum\limits_{k,l\in \mathbb{Z}^{n}}|\langle f, E_{k,l}\rangle|^{2},
	\end{equation*}
where $E_{k,l}(\c,\d)=e^{2\pi i(k\cdot\xi+l\cdot \d)}e^{\pi ik\cdot l}$.
\end{remark}

Given two functions $\phi, \psi\in L^{2}(\mathbb{R}^{2n}),$ we define the bracket $\left[\p,\psi\right]$ using their Weyl-Zak transform as follows.\[\left[\p,\psi\right](\c,\d)=\int_{\mathbb{R}^{n}}Z_{W} \p(\c,\d,\e)\overline{Z_{W} \psi(\c,\d,\e)}\,d\e.\] The right hand side integral is meaningful because
\begin{align*}\label{bracketwell}
\bigg|\int_{\mathbb{R}^{n}}Z_{W} \p(\c,\d,\e)\overline{Z_{W} \psi(\c,\d,\e)}\,d\e\bigg|&\leq \int_{\mathbb{R}^{n}}|Z_{W} \p(\c,\d,\e)||Z_{W} \psi(\c,\d,\e)\,|d\e\\
&\leq \bigg(\int_{\mathbb{R}^{n}}|Z_{W} \p(\c,\d,\e)|^{2}d\e\bigg)^{\frac{1}{2}}\bigg(\int_{\mathbb{R}^{n}}|Z_{W} \psi(\c,\d,\e)|^{2}d\e\bigg)^{\frac{1}{2}}<\infty,\numberthis
\end{align*}
by applying Cauchy-Schwartz inequality for $Z_{W} \p(\c,\d,\cdot), Z_{W} \psi(\c,\d,\cdot)\in L^{2}(\mathbb{R}^{n}),$ for a.e $\c,\d\in\T$. Clearly $\left[\cdot,\cdot\right]$ is a sesquilinear form and Hermitian symmetric map. Further, \begin{equation*}
\left[\p,\p\right](\c,\d)=\int_{\mathbb{R}^{n}}|Z_{W} \p(\c,\d,\e)|^{2}d\e\geq 0,
\end{equation*}
for a.e $\c,\d\in \mathbb{T}^{n},~\text{for all}~\p\in \l.$ Furthermore, 
\begin{equation*}\label{bracketpoint}
	\big|\left[\phi,\psi\right](\c,\d)\big|\leq \bigg(\big|\left[\phi,\phi\right](\c,\d)\big|\bigg)^{\frac{1}{2}}\bigg(\big|\left[\psi,\psi\right](\c,\d)\big|\bigg)^{\frac{1}{2}},\numberthis
\end{equation*}
for a.e $\c,\d\in \mathbb{T}^{n},~\text{for all}~\p,\psi\in \l,$ by using \eqref{bracketwell}. Now, consider 
\begin{align*}
\|\left[\p,\psi\right]\|_{L^{1}(\mathbb{T}^{n}\times\mathbb{T}^{n})}&= \int_{\mathbb{T}^{n}}\int_{\mathbb{T}^{n}}|\left[\p,\psi\right](\c,\d)|\,d\d d\c\\
&\leq \int_{\mathbb{T}^{n}}\int_{\mathbb{T}^{n}} (\left[\p,\p\right](\c,\d))^{\frac{1}{2}}(\left[\psi,\psi\right](\c,\d))^{\frac{1}{2}}\,d\d d\c\\
& \leq \bigg(\int_{\mathbb{T}^{n}}\int_{\mathbb{T}^{n}} \left[\p,\p\right](\c,\d)\,d\d\,d\c\bigg)^{\frac{1}{2}}\bigg(\int_{\mathbb{T}^{n}}\int_{\mathbb{T}^{n}} \left[\psi,\psi\right](\c,\d)\,d\d d\c\bigg)^{\frac{1}{2}}\\
&=\bigg(\int_{\mathbb{T}^{n}}\int_{\mathbb{T}^{n}} \int_{\mathbb{R}^{n}}|Z_{W}\p(\c,\d,\e)|^{2}\,d\e d\d d\c\bigg)^{\frac{1}{2}}\bigg(\int_{\mathbb{T}^{n}}\int_{\mathbb{T}^{n}}\int_{\mathbb{R}^{n}}|Z_{W}\psi(\c,\d,\e)|^{2}\,d\e d\d d\c\bigg)^{\frac{1}{2}}\\
&=\|Z_{W}\p\|_{\L}\|Z_{W}\psi\|_{\L}\\
&=\|\p\|_{\l}\|\psi\|_{\l},
\end{align*}
by applying Cauchy-Schwartz inequality and using \eqref{bracketpoint}.\\

In the following proposition we observe some interesting properties involving twisted translates of a function in terms of the bracket map. Let $V^{t}(\p)=\overline{\Span\{\t\p : k,l\in\mathbb{Z}^{n}\}},$ denote the principal twisted shift-invaraint space, where $\p\in L^{2}(\mathbb{R}^{2n}).$
\begin{proposition}
	For $k,l\in \mathbb{Z}^{n}$, $\p,\psi\in L^{2}(\mathbb{R}^{2n}),$ we have the following properties.
	\begin{itemize}
		\item [(i)] $\langle\p,T^{t}_{(k,l)}\psi\rangle_{\l}=\int_{\mathbb{T}^{n}}\int_{\mathbb{T}^{n}}\left[\p,\psi\right](\c,\d)e^{-2\pi i(k\cdot \c+l\cdot \d)}e^{-\pi ik\cdot l}d\d d\c$.
		\item [(ii)] $\left[\t\p,\psi\right](\c,\d)=e^{2\pi i(k\cdot\c+l\cdot\d)}e^{-\pi ik\cdot l}\left[\p,\psi\right](\c,\d)$ for a.e $\c,\d\in\mathbb{T}^{n}$.
		\item [(iii)] $\left[\p,\t\psi\right](\c,\d)=e^{-2\pi i(k\cdot\c+l\cdot\d)}e^{-\pi ik\cdot l}\left[\p,\psi\right](\c,\d)$ for a.e $\c,\d\in\mathbb{T}^{n}$.
		\item [(iv)] $\p\perp V^{t}(\psi)~\text{if and only if}~ \left[\p,\psi\right](\c,\d)=0~\text{a.e}~\c,\d\in\mathbb{T}^{n}$.
	\end{itemize}
\end{proposition}
\begin{proof} 
	\begin{itemize}
		\item [(i)] Consider 
		\begin{align*}\label{bracket}
			\langle\p,T^{t}_{(k,l)}\psi\rangle_{\l}&=\langle Z_{W}\p,Z_{W} T^{t}_{(k,l)}\psi\rangle_{\L}\\
			&=\int_{\mathbb{T}^{n}}\int_{\mathbb{T}^{n}}\int_{\mathbb{R}^{n}}Z_{W}\p(\c,\d,\e)\overline{Z_{W}\psi(\c,\d,\e)e^{2\pi i(k\cdot \c+l\cdot \d)}e^{\pi ik\cdot l}}\,d\e d\d d\c\\
			&=\int_{\mathbb{T}^{n}}\int_{\mathbb{T}^{n}}\bigg(\int_{\mathbb{R}^{n}}Z_{W}\p(\c,\d,\e)\overline{Z_{W}\psi(\c,\d,\e)}\,d\e\bigg)e^{-2\pi i(k\cdot \c+l\cdot \d)}e^{-\pi ik\cdot l}\,d\d d\c\\
			&=\int_{\mathbb{T}^{n}}\int_{\mathbb{T}^{n}}\left[\p,\psi\right](\c,\d)e^{-2\pi i(k\cdot \c+l\cdot \d)}e^{-\pi ik\cdot l}\,d\d d\c,\numberthis
		\end{align*}
		since $Z_{W}$ is unitary and by making use of \eqref{zaktwist}.
		\item [(ii)]
		Now, by making use of \eqref{comptsttrns} and \eqref{bracket}, we get
		\begin{align*}
			\langle \t\p,T^{t}_{(k^{'},l^{'})}\psi\rangle_{\l}&=\langle \p,T^{t}_{(-k,-l)}T^{t}_{(k^{'},l^{'})}\psi\rangle_{\l}\\
			&=e^{\pi i(k\cdot l^{'}-l\cdot k^{'})}\langle \p,T^{t}_{(k^{'}-k,l^{'}-l)}\psi\rangle_{\l}\\
			&=e^{\pi i(k\cdot l^{'}-l\cdot k^{'})}\\
			&\hspace{2mm}\times\int_{\mathbb{T}^{n}}\int_{\mathbb{T}^{n}}\left[\p,\psi\right](\c,\d)e^{-2\pi i((k^{'}-k)\cdot \c+(l^{'}-l)\cdot \d)}e^{-\pi i(k^{'}-k)\cdot (l^{'}-l)}\,d\d d\c\\
			&=\int_{\mathbb{T}^{n}}\int_{\mathbb{T}^{n}}e^{2\pi i(k\cdot \c+l\cdot \d)}e^{-\pi ik\cdot l}\left[\p,\psi\right](\c,\d)e^{-2\pi i(k^{'}\cdot \c+l^{'}\cdot \d)}e^{-\pi ik^{'}\cdot l^{'}}d\d d\c.
		\end{align*}
		On the other hand,
		\begin{align*}
			\langle \t\p,T^{t}_{(k^{'},l^{'})}\psi\rangle_{\l}&=\int_{\mathbb{T}^{n}}\int_{\mathbb{T}^{n}}\left[\t\p,\psi\right](\c,\d)e^{-2\pi i(k^{'}\cdot \c+l^{'}\cdot \d)}e^{-\pi ik^{'}\cdot l^{'}}d\d d\c,
		\end{align*}
	by using \eqref{bracket}.
		Thus
		\begin{equation}\label{bracketptwist}
			\left[\t\p,\psi\right](\c,\d)=e^{2\pi i(k\cdot \c+l\cdot \d)}e^{-\pi ik\cdot l}\left[\p,\psi\right](\c,\d),~\text{a.e}~\c,\d\in\T,
		\end{equation}
	by making use of Remark \ref{plancherel theorem}, proving $(ii).$
	\item [(iii)]
		Consider
		\begin{align*}\label{v2}
			\left[\p,\t\psi\right](\c,\d)&=\int_{\mathbb{R}^{n}}Z_{W}\p(\c,\d,\e)\overline{Z_{W}\t\psi(\c,\d,\e)}\,d\e\\
			&=\int_{\mathbb{R}^{n}}Z_{W}\p(\c,\d,\e)e^{-2\pi i(k\cdot \c+l\cdot \d)}e^{-\pi ik\cdot l}\overline{Z_{W}\psi(\c,\d,\e)}\,d\e\\
			&=\int_{\mathbb{R}^{n}}Z_{W}\p(\c,\d,\e)e^{-2\pi i(k\cdot \c+l\cdot \d)}e^{\pi ik\cdot l}e^{-2\pi ik\cdot l}\overline{Z_{W}\psi(\c,\d,\e)}\,d\e\\
			&=\int_{\mathbb{R}^{n}}Z_{W}\p(\c,\d,\e)e^{-2\pi i(k\cdot \c+l\cdot \d)}e^{\pi ik\cdot l}\overline{Z_{W}\psi(\c,\d,\e)}\,d\e\\
			&=\int_{\mathbb{R}^{n}}Z_{W} T^{t}_{(-k,-l)}\p(\c,\d,\e)\overline{Z_{W}\psi(\c,\d,\e)}\,d\e\\
			&=\left[T^{t}_{(-k,-l)}\p,\psi\right](\c,\d)\numberthis.
		\end{align*}
		Then the result follows by substituting \eqref{bracketptwist} in \eqref{v2}.
		\item [(iv)] Suppose $\p\perp V^{t}(\psi)$. Then 
		\begin{equation*}
			\langle \p,\t\psi\rangle_{\l}=0~\text{for all}~k,l\in \Z.
		\end{equation*}
		By \eqref{bracket}, it is equivalent to 
		\begin{equation*}
			\int_{\mathbb{T}^{n}}\int_{\mathbb{T}^{n}}\left[\p,\psi\right](\c,\d)e^{-2\pi i(k\cdot \c+l\cdot \d)}e^{-\pi ik\cdot l}d\d d\c=0~\text{for all}~k,l\in \Z,
		\end{equation*}
		which is equivalent to 
		\begin{equation*}
			\left[\p,\psi\right](\c,\d)=0~\text{a.e}~\c,\d\in\T,
		\end{equation*}
	by using Remark \ref{plancherel theorem}, proving (iv).
	\end{itemize}
\end{proof}
\section{A frame sequence and a Riesz sequence of the system of twisted translates}
Let $\p\in\l$, $E^{t}(\p):=\{\t \p : k,l\in\Z\}$ and $U^t(\p)=\Span( E^{t}(\p))$. The following proposition gives the structure of elements in image of $V^{t}(\p)$ under Weyl-Zak transform, for some $\p\in L^{2}(\mathbb{R}^{2n}).$
\begin{proposition}\label{propo}
	Let $\p\in \l$. Then $f\in V^{t}(\p)$ if and only if
	\begin{equation*}
		Z_{W}f(\c,\d,\e)=r(\c,\d)Z_{W}\p(\c,\d,\e),~\text{for a.e}~\c,\d\in \T, \e\in \mathbb{R}^{n},
	\end{equation*}
	for some $r\in L^{2}(\T \times \T;\left[\p,\p\right]).$
\end{proposition}
\begin{proof}
	Let $f\in U^{t}(\p).$ Then $f=\sum\limits_{(k,l)\in \mathcal{F}}a_{k,l}\t\p,$ where $\mathcal{F}$ is a finite subset of $\mathbb{Z}^{2n}.$ By using \eqref{zaktwist}, we get
	\begin{align*}
		Z_{W} f(\c,\d,\e)&=\sum\limits_{(k,l)\in \mathcal{F}}a_{k,l}Z_{W}\t\p(\c,\d,\e)\\
		&=\sum\limits_{(k,l)\in \mathcal{F}}a_{k,l}e^{2\pi i(k\cdot \c+l\cdot \d)}e^{\pi ik\cdot l}Z_{W}\p(\c,\d,\e)\\
		&=\bigg(\sum\limits_{(k,l)\in \mathcal{F}}a_{k,l}e^{2\pi i(k\cdot \c+l\cdot \d)}e^{\pi ik\cdot l}\bigg)Z_{W}\p(\c,\d,\e)\\
		&=:r_{f}(\c,\d)Z_{W}\p(\c,\d,\e).
	\end{align*}
Further,
	\begin{align*}\label{propiso}
		\|f\|^{2}_{\l}&=\|Z_{W} f\|^{2}_{\L}\\
		&=\int_{\mathbb{T}^{n}}\int_{\mathbb{T}^{n}}|r_{f}(\c,\d)|^{2}\int_{\mathbb{R}^{n}}|Z_{W} \p(\c,\d,\e)|^{2}\,d\e \,d\d\,d\c\\
		&=\int_{\mathbb{T}^{n}}\int_{\mathbb{T}^{n}}|r_{f}(\c,\d)|^{2}\left[\p,\p\right](\c,\d)\,d\d\,d\c\\
		&=\|r_{f}\|^{2}_{L^{2}(\T\times \T;\left[\p,\p\right])}.\numberthis
	\end{align*}
Then the map $f\mapsto r_{f}$ is an isometric isomorphism of $U^{t}(\p)$ onto the collection of all trigonometric polynomials in $L^{2}(\T\times \T;\left[\p,\p\right])$, which can be extended to an isometric isomorphism between $V^{t}(\p)$ and $L^{2}(\T\times \T;\left[\p,\p\right]).$
\end{proof}
\begin{theorem}
	Let $\p\in \l.$ Then $\{\t\p : k,l\in \Z\}$ is a frame sequence with frame bounds $A,B>0$ iff 
	\begin{equation*}\label{frameif}
		0<A\leq \left[\p,\p\right](\c,\d)\leq B<\infty,\hspace{5mm}\text{for a.e}~(\c,\d)\in\Omega_{\p}\numberthis,
	\end{equation*}
	where $\Omega_{\p}:=\{(\c,\d)\in \mathbb{T}^{n}\times\T : \left[\p,\p\right](\c,\d)\neq 0\}.$
\end{theorem}
\begin{proof}
	Assume that $\{\t\p : k,l\in \Z\}$ is a frame sequence. Then 
	\begin{equation*}\label{framedef}
		A\|f\|^{2}_{\l}\leq \sum\limits_{k,l\in\mathbb{Z}^{n}}|\langle f,\t\p\rangle_{\l}|^{2}\leq B\|f\|^{2}_{\l},\hspace{5mm}\forall~f\in V^{t}(\p)\numberthis.
	\end{equation*}
	Consider,
	\begin{align*}
		\langle f,\t\p\rangle_{\l}&=\langle Z_{W} f,Z_{W}\t\p\rangle_{L^{2}(\T\times \T\times \mathbb{R}^{n})}\\
		&=\int_{\mathbb{T}^{n}}\int_{\mathbb{T}^{n}}\int_{\mathbb{R}^{n}}r(\c,\d)Z_{W} \p(\c,\d,\e)e^{-2\pi i(k\cdot \c+l\cdot \d)}e^{-\pi i k\cdot l}\overline{Z_{W}\p(\c,\d,\e)}\,d\e d\d d\c,
	\end{align*}
by using Proposition \ref{zaktwistprop} and Proposition \ref{propo}. Thus
		\begin{align*}
			\langle f,\t\p\rangle_{\l}=\int_{\mathbb{T}^{n}}\int_{\mathbb{T}^{n}}r(\c,\d)\left[\p,\p\right](\c,\d)e^{-2\pi i(k\cdot \c+l\cdot \d)}e^{-\pi i k\cdot l}\,d\d d\c.
		\end{align*}
Now, by using Remark \ref{plancherel theorem}, we have
	\begin{align*}\label{framemid}
		\sum\limits_{k,l\in\mathbb{Z}^{n}}|\langle f,\t\p\rangle_{\l}|^{2}&=\sum\limits_{k,l\in\mathbb{Z}^{n}}\bigg|\int_{\mathbb{T}^{n}}\int_{\mathbb{T}^{n}}r(\c,\d)\left[\p,\p\right](\c,\d)e^{-2\pi i(k\cdot \c+l\cdot \d)}e^{-\pi k\cdot l}\,d\d d\c\bigg|^{2}\\
		&=\|r\left[\p,\p\right]\|^{2}_{L^{2}(\T\times \T)}\\
		&=\int_{\mathbb{T}^{n}}\int_{\mathbb{T}^{n}}|r(\c,\d)\left[\p,\p\right](\c,\d)|^{2}\,d\d d\c.\numberthis
	\end{align*}
	Substituting \eqref{propiso} and \eqref{framemid} in \eqref{framedef}, we get 
	\begin{equation*}\label{L.H.S}
		A\|r\|^{2}_{L^{2}(\T\times \T;\left[\p,\p\right])}\leq \int_{\mathbb{T}^{n}}\int_{\mathbb{T}^{n}}|r(\c,\d)\left[\p,\p\right](\c,\d)|^{2}\,d\d d\c\numberthis
	\end{equation*}
	and
	\begin{equation*}\label{R.H.S}
		\int_{\mathbb{T}^{n}}\int_{\mathbb{T}^{n}}|r(\c,\d)\left[\p,\p\right](\c,\d)|^{2}\,d\d d\c\leq B\|r\|^{2}_{L^{2}(\T\times \T;\left[\p,\p\right])}\numberthis.
	\end{equation*}
	Let $N:=\{(\c,\d)\in \Omega_{\p} : \left[\p,\p\right](\c,\d)-B>0\}$ be a measurable subset of $\T\times \T$. By choosing $r=\chi_{N}$ in \eqref{R.H.S}, we obtain
	\begin{align*}
		\iint\limits_{N}\left[\p,\p\right](\c,\d)&(\left[\p,\p\right](\c,\d)-B)\,d\d d\c\\
		&=\iint\limits_{\Omega_{\p}}|r(\c,\d)|^{2}\left[\p,\p\right](\c,\d)(\left[\p,\p\right](\c,\d)-B)\,d\d d\c\\
		&\leq 0, 
	\end{align*}
	which implies that $\mu(N)=0,$ where $\mu$ is the Lebesgue measure $d\c\,d\d$ on $\T\times \T.$ Similarly, we can show that $A\leq \left[\p,\p\right](\c,\d)$ for a.e $(\c,\d)\in\Omega_{\p}$, which proves our assertion. Conversely, assume that \eqref{frameif} holds. Then we have
	\begin{equation*} \iint\limits_{\Omega_{\p}}|r(\c,\d)|^{2}\left[\p,\p\right](\c,\d)(A-\left[\p,\p\right](\c,\d))\,d\d\,d\c\leq 0
	\end{equation*}
	and
	\begin{equation*}
		\iint\limits_{\Omega_{\p}}|r(\c,\d)|^{2}\left[\p,\p\right](\c,\d)(\left[\p,\p\right](\c,\d)-B)\,d\d\,d\c\leq 0.
	\end{equation*}
Now retracing the steps back, we get $\{\t\p : k,l\in \mathbb{Z}^{n}\}$ is a frame sequence with frame bounds $A$ and $B$.
\end{proof}
\begin{theorem}\label{riesz}
	Let $\p\in\l$. Then $\{\t \p : k,l\in\Z\}$ is a Riesz sequence for $V^{t}(\p)$ if and only if
	\begin{equation*}\label{rieszequ}
		0<A\leq\left[\p,\p\right](\c,\d)\leq B<\infty~\text{a.e}~\c,\d\in\T\numberthis.
	\end{equation*}
\end{theorem}
\begin{proof}
	Let $\{a_{k,l}\}\in \ell^{2}(\mathbb{Z}^{2n}).$ By using \eqref{zaktwist}, we get 
	\begin{align*}\label{reisz1}
		\bigg\|\sum\limits_{k,l\in\mathbb{Z}^{n}}a_{k,l}\t \phi\bigg\|^{2}_{L^{2}(\mathbb{R}^{2n})}&=\bigg\|Z_{W}\bigg(\sum\limits_{k,l\in\mathbb{Z}^{n}}a_{k,l}\t \phi\bigg)\bigg\|^{2}_{\L}\\
		&=\bigg\|\sum\limits_{k,l\in\mathbb{Z}^{n}}a_{k,l}Z_{W}\t \phi\bigg\|^{2}_{\L}\\
		&=\int_{\mathbb{T}^{n}}\int_{\mathbb{T}^{n}}\int_{\mathbb{R}^{n}}\bigg|\sum\limits_{k,l\in\mathbb{Z}^{n}}a_{k,l}Z_{W}\t \phi(\c,\d,\e)\bigg|^{2}d\e d\d d\c\\
		&=\int_{\mathbb{T}^{n}}\int_{\mathbb{T}^{n}}\int_{\mathbb{R}^{n}}\bigg|\sum\limits_{k,l\in\mathbb{Z}^{n}}a_{k,l}e^{2\pi i(k\cdot \c+l\cdot \d)}e^{\pi ik\cdot l}Z_{W}\phi(\c,\d,\e)\bigg|^{2}d\e d\d d\c\\
		&=\int_{\mathbb{T}^{n}}\int_{\mathbb{T}^{n}}\bigg|\sum\limits_{k,l\in\mathbb{Z}^{n}}a_{k,l}e^{2\pi i(k\cdot \c+l\cdot \d)}e^{\pi ik\cdot l}\bigg|^{2}\int_{\mathbb{R}^{n}}\bigg|Z_{W}\phi(\c,\d,\e)\bigg|^{2}d\e d\d d\c\\
		&=\int_{\mathbb{T}^{n}}\int_{\mathbb{T}^{n}}\bigg|\sum\limits_{k,l\in\mathbb{Z}^{n}}a_{k,l}e^{2\pi i(k\cdot \c+l\cdot \d)}e^{\pi ik\cdot l}\bigg|^{2}\left[\p,\p\right](\c,\d)d\d d\c.\numberthis
	\end{align*}
	Let $p\in L^{2}(\T\times \T),$ then 
	\begin{equation*}\label{reisz2}
		p(\c,\d)=\sum\limits_{k,l\in \Z}a_{k,l}e^{2\pi i(k\cdot \c+l\cdot \d)}e^{\pi ik\cdot l}\hspace{5mm}\text{and}\numberthis
	\end{equation*}
	\begin{equation}\label{reisz21}
		\|p\|^{2}_{L^{2}(\T\times \T)}=\sum\limits_{k,l\in\mathbb{Z}^{n}}|a_{k,l}|^{2},
	\end{equation}
	by Remark \ref{plancherel theorem}. Then it follows from \eqref{reisz1} that 
	\begin{equation*}
		\bigg\|\sum\limits_{k,l\in\mathbb{Z}^{n}}a_{k,l}\t \phi\bigg\|^{2}_{L^{2}(\mathbb{R}^{2n})}=\int_{\mathbb{T}^{n}}\int_{\mathbb{T}^{n}}|p(\c,\d)|^{2}\left[\p,\p\right](\c,\d)d\d d\c.
	\end{equation*}
	Assume that $\{\t\p : k,l\in\mathbb{Z}^{n}\}$ is a Riesz sequence. Then there exists $A,B>0$ such that  
	\begin{equation*}\label{rif}
		A\sum\limits_{k,l\in\mathbb{Z}^{n}}|a_{k,l}|^{2}\leq \bigg\|\sum\limits_{k,l\in\mathbb{Z}^{n}}a_{k,l}\t \phi\bigg\|^{2}_{L^{2}(\mathbb{R}^{2n})} \leq B\sum\limits_{k,l\in\mathbb{Z}^{n}}|a_{k,l}|^{2},~~\forall~\{a_{k,l}\}\in \ell^{2}(\mathbb{Z}^{2n})\numberthis.
	\end{equation*}
	Thus
	\begin{align*}\label{rif}
		A\int_{\mathbb{T}^{n}}\int_{\mathbb{T}^{n}}|p(\c,\d)|^{2}\,d\d d\c\leq 
		\int_{\mathbb{T}^{n}}\int_{\mathbb{T}^{n}}|p(\c,\d)|^{2}\left[\p,\p\right](\c,\d)\,d\d d\c\leq B\int_{\mathbb{T}^{n}}\int_{\mathbb{T}^{n}}|p(\c,\d)|^{2}\,d\d d\c,\numberthis
	\end{align*}
	for any $p\in L^{2}(\T\times \T).$ Let $M:=\{(\c,\d)\in \mathbb{T}^{n}\times \T : \left[\p,\p\right](\c,\d)>B \}$ be a measurable subset of $\mathbb{T}^{n}\times \T.$ By taking $p=\chi_{M}$ in R.H.S of \eqref{rif}, we get 
	\begin{equation*}
		\iint\limits_{M}(\left[\p,\p\right](\c,\d)-B)\,d\d d\c\leq 0,
	\end{equation*}
	which implies that $\mu(M)=0,$ where $\mu$ is the Lebesgue measure $d\c\, d\d~\text{on}~\T\times \T.$ Similarly, we can show that $\mu\{(\c,\d)\in \mathbb{T}^{n}\times \T : \left[\p,\p\right](\c,\d)<A\}=0$. Thus
	\[0<A\leq\left[\p,\p\right](\c,\d)\leq B<\infty~\text{a.e}~\c,\d\in\T.\]
	Conversely, assume that \eqref{rieszequ} holds. Let $\{a_{k,l}\}\in C_{00}(\mathbb{Z}^{2n}).$ Then by using \eqref{reisz1}, we get 
	\begin{align*}\label{ronly}
		A\int_{\mathbb{T}^{n}}\int_{\mathbb{T}^{n}}\bigg|\sum\limits_{k,l\in\mathbb{Z}^{n}}a_{k,l}e^{2\pi i(k\cdot \c+l\cdot \d)}e^{\pi ik\cdot l}\bigg|^{2}\,d\d d\c&\leq \bigg\|\sum\limits_{k,l\in\mathbb{Z}^{n}}a_{k,l}\t \phi\bigg\|^{2}_{L^{2}(\mathbb{R}^{2n})}\\
		&\leq B\int_{\mathbb{T}^{n}}\int_{\mathbb{T}^{n}}\bigg|\sum\limits_{k,l\in\mathbb{Z}^{n}}a_{k,l}e^{2\pi i(k\cdot \c+l\cdot \d)}e^{\pi ik\cdot l}\bigg|^{2}\,d\d d\c.\numberthis
	\end{align*}
	Then the result follows by using \eqref{reisz21}.
\end{proof}
We illustrate the theorem with an example.
\begin{example}
	Define $\p\in L^{2}(\mathbb{R}^{2})$ by \[K_{\p}(\c,\e)=e^{\c\e}\chi_{[0,1)}(\c)\chi_{[0,1)}(\e),\hspace{5mm}\c,\e \in\mathbb{R}.\] 
	Let $\p_{2}=\p\times \p,$ where $\times$ is a twisted convolution. Now, we aim to show that $\{\t\p_{2} : k,l\in \mathbb{Z}\}$ is a Riesz sequence, by using Theorem \ref{riesz}. We know that the kernel of the Weyl transform of $\p_{2}$ is given by $K_{\p_{2}}(\c,\e)=\int_{\mathbb{R}}K_{\p}(\c,y)K_{\p}(y,\e)\,dy.$ Now, for $\c,\d\in \mathbb{T}$ and $\e\in \mathbb{R}$, we get
	\begin{align*}
		Z_{W}\p(\c,\d,\e)&=\sum\limits_{m\in \mathbb{Z}}K_{\p}(\c+m,\e)e^{-2\pi im\d}\\
		&=K_{\p}(\c,\e)\\
		&=e^{\c\e}\chi_{[0,1)}(\e)
	\end{align*}
	and 
	\begin{align*}
		Z_{W}\p_{2}(\c,\d,\e)&=\sum\limits_{m\in \mathbb{Z}}K_{\p_{2}}(\c+m,\e)e^{-2\pi im\d}\\
		&=\sum\limits_{m\in \mathbb{Z}}\bigg(\int_{\mathbb{R}}K_{\p}(\c+m,y)K_{\p}(y,\e)\,dy\bigg)e^{-2\pi im\c}\\
		&=\int_{\mathbb{R}}\bigg(\sum\limits_{m\in \mathbb{Z}}K_{\p}(\c+m,y)e^{-2\pi im\c}\bigg)K_{\p}(y,\e)\,dy\\
		&=\int_{\mathbb{R}}Z_{W}\p(\c,\d,y)K_{\p}(y,\e)\,dy\\
		&=\int_{\mathbb{R}}e^{\c y}\chi_{[0,1)}(y)e^{y\e}\chi_{[0,1)}(y)\chi_{[0,1)}(\e)\,dy\\
		&=\chi_{[0,1)}(\e)\int_{0}^{1}e^{(\c+\e)y}\,dy.
	\end{align*}
	Further,
	\begin{align*}
		\int_{\mathbb{R}}|Z_{W}\p_{2}(\c,\d,\e)|^{2}\,d\e
		&=\int_{\mathbb{R}}\bigg|\chi_{[0,1)}(\e)\int_{0}^{1}e^{(\c+\e)y}\,dy\bigg|^{2}\,d\e\\
		&=\int_{0}^{1}\int_{0}^{1}e^{(\c+\e)y}\,dyd\e.
	\end{align*}
	Since $1\leq e^{(\c+\e)y}<e^{2},~\text{for all}~\c,\e,y\in [0,1)$,  we have
	\begin{align*}
		1\leq \int_{\mathbb{R}}|Z_{W}\p_{2}(\c,\d,\e)|^{2}\,d\e \leq e^{2},
	\end{align*}
	for all $\c,\d\in \mathbb{T}$.
	Therefore the collection $\{\t \p_{2} : k,l\in \mathbb{Z}\}$ forms a Riesz basis for $V^{t}(\p\times \p).$
\end{example}

\section{Canonical biorthogonal function and orthogonalization}
 The following theorem gives an equivalent condition for the system $\{\t \p : k,l\in\Z\}$ to be an orthonormal system.
\begin{theorem}\label{ortho}
	Let $\p\in\l$. Then $\{\t \p : k,l\in\Z\}$ is an orthonormal system if and only if $\left[\p,\p\right](\c,\d)=1~\text{a.e}~\c,\d\in\T.$
\end{theorem}
\begin{proof}
	From \eqref{comptsttrns} and \eqref{bracket}, we observe that 
	\[\langle\t\p,T^{t}_{(k^{'},l^{'})}\p\rangle_{\l}=\delta_{(k,l),(k^{'},l^{'})}\hspace{5mm}\text{for}~k,k^{'},l,l^{'}\in\Z,\] if and only if 
	\[\langle\p,T^{t}_{(k,l)}\p\rangle_{\l}=\delta_{(k,l),(0,0)}\hspace{5mm}\text{for}~k,l\in\Z,\] if and only if
	\begin{align*}
		\langle\p,T^{t}_{(k,l)}\p\rangle_{\l}&=\int_{\mathbb{T}^{n}}\int_{\mathbb{T}^{n}}\left[\p,\p\right](\c,\d)e^{-2\pi i(k\cdot \c+l\cdot \d)}e^{-\pi ik\cdot l}d\d d\c.
	\end{align*}
Thus if $\{\t \p : k,l\in\Z\}$ is an orthonormal system, then it follows from Remark \ref{plancherel theorem} that
	\[\left[\p,\p\right](\c,\d)=1,\hspace{5mm}~\text{for a.e}~\c,\d\in \T.\]
Converse is obtained by retracing the steps.
\end{proof}
\begin{theorem}
	Let $\p\in \l.$ Then there exists a canonical biorthogonal function $\tilde{\p}$ to $\p$ such that $\tilde{\p}\in V^{t}(\p)$ if and only if $\frac{1}{\left[\p,\p\right]}\in L^{1}(\T\times \T).$ 
\end{theorem}
\begin{proof}
	Suppose $\tilde{\p}\in V^{t}(\p)$ is a canonical biorthogonal function to $\p$ (biorthogonality is taken with respect to the system of twisted translates). Then, we have 
	\begin{equation*}
		\langle \t\tilde{\p},\p\rangle_{\l}=\delta_{(k,l),(0,0)},\hspace{5mm}k,l\in \Z.
	\end{equation*}
	By Proposition \ref{zaktwistprop} and Proposition \ref{propo}, we get 
	\begin{align*}\label{biortho}
		\delta_{(k,l),(0,0)}&=\langle \t\tilde{\p},\p\rangle_{\l}\\
		&=\int_{\mathbb{T}^{n}}\int_{\mathbb{T}^{n}}\int_{\mathbb{R}^{n}}e^{2\pi i(k\cdot \c+l\cdot \d)}e^{\pi ik\cdot l}Z_{W}\tilde{\p}(\c,\d,\e)\overline{Z_{W}\p(\c,\d,\e)}\,d\e\,d\d\,d\c\\
		&=\int_{\mathbb{T}^{n}}\int_{\mathbb{T}^{n}}\int_{\mathbb{R}^{n}}e^{2\pi i(k\cdot \c+l\cdot \d)}e^{\pi ik\cdot l}r(\c,\d)Z_{W}\p(\c,\d,\e)\overline{Z_{W}\p(\c,\d,\e)}\,d\e\,d\d\,d\c\\
		&=\int_{\mathbb{T}^{n}}\int_{\mathbb{T}^{n}}e^{2\pi i(k\cdot \c+l\cdot \d)}e^{\pi ik\cdot l}r(\c,\d)\int_{\mathbb{R}^{n}}Z_{W}\p(\c,\d,\e)\overline{Z_{W}\p(\c,\d,\e)}\,d\e\,d\d\,d\c.\numberthis
	\end{align*}
	Then
	\begin{equation*}
		\delta_{(k,l),(0,0)}=\int_{\mathbb{T}^{n}}\int_{\mathbb{T}^{n}}r(\c,\d)\left[\p,\p\right](\c,\d)e^{2\pi i(k\cdot \c+l\cdot \d)}e^{\pi ik\cdot l}d\d\,d\c,
	\end{equation*}
	by making use of the definition of the bracket $\left[\p,\p\right],$ it follows that 
	\[r(\c,\d)\left[\p,\p\right](\c,\d)=1,\hspace{5mm}\text{for a.e}~\c,\d\in \T,\] using the orthonormal basis $\{e^{2\pi i(k\cdot \c+l\cdot \d)}e^{\pi ik\cdot l} : k,l\in \mathbb{Z}^{n}\}$.
	Thus, 
	\begin{align*}
		\int_{\mathbb{T}^{n}}\int_{\mathbb{T}^{n}}\frac{1}{\left[\p,\p\right](\c,\d)}\,d\d d\c&=\int_{\mathbb{T}^{n}}\int_{\mathbb{T}^{n}}\frac{|r(\c,\d)\left[\p,\p\right](\c,\d)|^{2}}{\left[\p,\p\right](\c,\d)}\,d\d d\c\\
		&=\int_{\mathbb{T}^{n}}\int_{\mathbb{T}^{n}}|r(\c,\d)|^{2}\left[\p,\p\right](\c,\d)\,d\d d\c\\
		&=\|r\|^{2}_{L^{2}(\T\times \T;\left[\p,\p\right](\c,\d))}<\infty,
	\end{align*}
	which proves our assertion. Conversely, assume that $\frac{1}{\left[\p,\p\right]}\in L^{1}(\T\times \T)$. Then  $\frac{1}{\left[\p,\p\right]}\in L^{2}(\T\times\T;\left[\p,\p\right])$. Define a function $\tilde{\p}$ such that $Z_{W}\tilde{\p}(\c,\d,\e)=\frac{1}{\left[\p,\p\right](\c,\d)}Z_{W}\p(\c,\d,\e).$ Then it follows from Proposition \ref{propo}, that $\tilde{\p}\in V^{t}(\p)$. Appealing to \eqref{biortho} and taking $\frac{1}{\left[\p,\p\right](\xi,\d)}$ in the place of $r(\c,\d)$, we get 
	\[\langle \t\tilde{\p},\p\rangle_{\l}=\int_{\mathbb{T}^{n}}\int_{\mathbb{T}^{n}}e^{2\pi i(k\cdot \c+l\cdot \d)}e^{\pi ik\cdot l}d\d\,d\c,\]
	from which we conclude that $\tilde{\p}$ is a canonical biorthogonal function to $\p$.
\end{proof}
In the forthcoming result, we prove that given a Riesz sequence of twisted translates $\{\t \p : k,l\in \mathbb{Z}^{n}\}$, we can find an appropriate function $\tilde{\p}\in V^{t}(\p)$ such that the resulting sequence forms an orthonormal sequence for $V^{t}(\p)$.
\begin{theorem}
	Let $\p\in L^{2}(\mathbb{R}^{2n})$ such that the system $\{\t \p : k,l\in \mathbb{Z}^{n}\}$ forms a Riesz sequence. Then there exists $\tilde{\p}\in V^{t}(\p)$ such that $\{\t \tilde{\p} : k,l\in \mathbb{Z}^{n}\}$ is an orthonormal system for $V^{t}(\p)$ with $V^{t}(\p)=V^{t}(\tilde{\p})$. More explicitly, $\tilde{\p}$ is given by $Z_{W}\tilde{\p}(\c,\d,\e)=\frac{1}{\sqrt{\left[\p,\p\right](\c,\d)}}Z_{W}\p(\c,\d,\e).$   
\end{theorem}
\begin{proof}
	Since $\{\t \p : k,l\in \mathbb{Z}^{n}\}$ is a Riesz sequence, by Theorem \ref{riesz} there exist $A,B>0$ such that $A\leq \left[\p,\p\right](\c,\d)\leq B$, for $a.e.~\c,\d \in \mathbb{T}^{n}$. We define $\tilde{\p}\in L^{2}(\mathbb{R}^{2n})$ by
	\begin{equation*}
		Z_{W}\tilde{\p}(\xi,\d,\eta):=\frac{Z_{W}\p(\xi,\d,\eta)}{\sqrt{\left[\p,\p\right](\c,\d)}},
	\end{equation*}
	for $\eta\in\mathbb{R}^n$ and for $a.e.\ \xi,\d\in \T$. Further,
	\begin{align*}
		\left[\tilde{\p},\tilde{\p}\right](\c,\d)&=\int_{\mathbb{R}^{n}}|Z_{W}\tilde{\p}(\c,\d,\e)|^{2}\,d\e\\
		&=\int_{\mathbb{R}^{n}}\bigg|\frac{Z_{W}\p(\c,\d,\e)}{\sqrt{\left[\p,\p\right](\c,\d)}}\bigg|^{2}\,d\e\\
		&=\frac{1}{\left[\p,\p\right](\c,\d)}\int_{\mathbb{R}^{n}}|Z_{W}\p(\c,\d,\e)|^{2}\,d\e\\
		&=\frac{1}{\left[\p,\p\right](\c,\d)}\left[\p,\p\right](\c,\d)\\
		&=1,
	\end{align*}
	for $a.e.\ \xi,\d \in \T$. Hence, by Theorem \ref{ortho}, the system $\{\t \tilde{\p} : k,l\in \mathbb{Z}^{n}\}$ is an orthonormal system. Now, it remains to prove that $V^{t}(\p)=V^{t}(\tilde{\p})$. Let $r(\c,\d)=\frac{1}{\sqrt{\left[\p,\p\right](\c,\d)}}$. Then $r\in L^{2}(\T\times \T;\left[\p,\p\right])$ and we have \[Z_{W}\tilde{\p}(\c,\e,\e)=r(\c,\d)Z_{W}\p(\c,\d,\e).\] By making use of Proposition \ref{propo}, we get $\tilde{\p}\in V^{t}(\p)$, which implies that $V^{t}(\tilde{\p})\subset V^{t}(\p)$. In order to prove the reverse inclusion, define $r_{1}(\c,\d)=\sqrt{\left[\p,\p\right](\c,\d)}$. Then $r_{1}\in L^{2}(\T\times \T;\left[\p,\p\right])$	as $\left[\p,\p\right](\c,\d)\leq B~\text{a.e}~\c,\d\in \mathbb{T}^{n}$ and \[Z_{W}\p(\c,\e,\e)=r_{1}(\c,\d)Z_{W}\tilde{\p}(\c,\d,\e).\] Then proceeding as before we get $V^{t}(\p)\subset V^{t}(\tilde{\p})$.
\end{proof}
\section{The system of twisted translates as a Schauder basis}
In this section, we take $n=1$ and give the characterization for the system $\{\t\p : k,l\in \mathbb{Z}\}$ to be a Schauder basis for $V^{t}(\p)$, for $\p\in L^{2}(\mathbb{R}^{2})$ in terms of Muckenhoupt $\mathcal{A}_{2}$ weight. We refer to \cite{wheeden} in this connection. Now, we state the following lemma for proving the main theorem. 
\begin{lemma}\label{Gaborequivalence}
	Let $w\in L^{1}(\mathbb{T}\times \mathbb{T})$ be a non negative function. For $M,N\geq 0$, define an operator $T_{M,N}$ on $L^{1}(\mathbb{T}\times \mathbb{T}; w)$ by $T_{M,N}f=\sum\limits_{m=-M}^{M}\sum\limits_{n=-N}^{N}\langle f, e_{m,n}\rangle e_{m,n},$ where $e_{m,n}(\c,\d)=e^{-2\pi i(m \c+n \d)}.$ Then the following statements are equivalent. 
	\begin{itemize}
		\item [(i)] $\sup\limits_{M,N}\|T_{M,N}\|<\infty.$
		\item [(ii)] $w\in \mathcal{A}_{2}(\mathbb{T}\times \mathbb{T}).$
	\end{itemize}
\end{lemma}
\begin{proof}
	The proof follows the same steps as in the proof of the Theorem 5.6 in \cite{Heil}.
\end{proof}
\begin{theorem}
	Let $\p\in L^{2}(\mathbb{R}^{2}).$ Then the system $\{\t\p : k,l\in \mathbb{Z}\}$ is a Schauder basis for $V^{t}(\p)$ if and only if $\left[\p,\p\right]\in \mathcal{A}_{2}(\mathbb{T}\times \mathbb{T}).$
\end{theorem} 
\begin{proof}
	By using the isometric isomorphism in Proposition \ref{propo}, we can show that $\{\t\p : k,l\in \mathbb{Z}\}$ is a Schauder basis for $V^{t}(\p)$ if and only if $\{e_{m,n}e^{-\pi imn} : m,n\in \mathbb{Z}\}$ is a Schauder basis for $L^{2}(\mathbb{T}\times \mathbb{T};\left[\p,\p\right]).$ Assume that $\{e_{m,n}e^{-\pi imn} : m,n\in \mathbb{Z}\}$ is a Schauder basis for $L^{2}(\mathbb{T}\times \mathbb{T};\left[\p,\p\right])$. Then, for each $f\in L^{2}(\mathbb{T}\times \mathbb{T};\left[\p,\p\right])$ there exists unique scalar sequence $\{c_{m,n}(f) : m,n\in \mathbb{Z}\}$ such that 
	\begin{equation*}
		f=\sum\limits_{m,n\in \mathbb{Z}}c_{m,n}(f)e_{m,n}e^{-\pi imn}.
	\end{equation*}
	Further, $f\rightarrow c_{m,n}(f)$ is a bounded linear functional, for every $m,n\in \mathbb{Z}.$ By Riesz representation theorem, there exists a collection $\{F_{m,n} : m,n\in \mathbb{Z}\}\subset\Ll$ such that
	\begin{equation*}
		c_{m,n}(f)=\langle f, F_{m,n}\rangle_{\Ll}.
	\end{equation*}
Notice that 
\begin{equation}\label{expreofbioortho}
	\langle e_{m^{'},n^{'}}e^{-\pi im^{'}n^{'}}, F_{m,n}\rangle_{\Ll}=\delta_{(m,n),(m^{'},n^{'})}.
\end{equation}
Thus $\{F_{m,n} : m,n\in \mathbb{Z}\}$ is biorthogonal to $\{e_{m,n}e^{-\pi imn} : m,n\in \mathbb{Z}\}.$ Now, for $M,N\geq 0,$ the partial sum operator $S_{M,N}$ on $\Ll$, defined by
	\begin{equation}\label{partialsum}
		S_{M,N}F=\sum\limits_{m=-M}^{M}\sum\limits_{n=-N}^{N}\langle F, F_{m,n}\rangle_{\Ll}e_{m,n}e^{-\pi imn}
	\end{equation}
	is uniformly bounded, by using Theorem \ref{schauderchara}. That is, 
	\begin{equation}\label{smnfinite}
		\sup\limits_{M,N\geq 0}\|S_{M,N}\|<\infty.
	\end{equation}	
 Thus, by using Cauchy-Schwartz inequality, we have
	\begin{align*}
		\int_{\mathbb{T}}\int_{\mathbb{T}}&|F_{m,n}(\c,\d)\left[\p,\p\right](\c,\d)|\,d\d d\c\\
		&\leq \bigg(\int_{\mathbb{T}}\int_{\mathbb{T}}|F_{m,n}(\c,\d)\sqrt{\left[\p,\p\right](\c,\d)}|^{2}\,d\d d\c\bigg)^{\frac{1}{2}}\bigg(\int_{\mathbb{T}}\int_{\mathbb{T}}|\sqrt{\left[\p,\p\right](\c,\d)}|^{2}\,d\d d\c\bigg)^{\frac{1}{2}}\\
		&=\|F_{m,n}\|_{\Ll}\|\left[\p,\p\right]\|_{L^{1}(\mathbb{T}\times \mathbb{T})}\\
		&<\infty.
	\end{align*}
	Hence $F_{m,n}\left[\p,\p\right]\in L^{1}(\mathbb{T}\times \mathbb{T})$, for each $m,n\in \mathbb{Z}.$ Further, by making use of \eqref{expreofbioortho},  we have
	\begin{align*}
		\int_{\mathbb{T}}\int_{\mathbb{T}}\overline{F_{m,n}(\c,\d)}\left[\p,\p\right](\c,\d)&e^{-2\pi i(m^{'}\c+n^{'}\d)}e^{-\pi im^{'}n^{'}}\,d\d\,d\c\\
		&=\langle e_{m^{'},n^{'}}e^{-\pi im^{'}n^{'}}, F_{m,n}\rangle_{\Ll}\\
		&=\delta_{(m,n),(m^{'},n^{'})}.
	\end{align*}
	Therefore $F_{m,n}\left[\p,\p\right]=e_{m,n}e^{-\pi imn}~\text{a.e}~\c,\d \in\mathbb{T}$, by using Remark \ref{plancherel theorem}. Let $T_{M,N}$ be an operator on $\Ll$ defined by \[T_{M,N}F=\sum\limits_{m=-M}^{M}\sum\limits_{n=-N}^{N}\langle F, e_{m,n}\rangle_{L^{2}(\mathbb{T}\times\mathbb{T})} e_{m,n}.\]
	Now, for any $F\in \Ll$ and for $m,n\in \mathbb{Z}$, we have
	\begin{align*}\label{tmn}
		\langle F, F_{m,n}\rangle_{\Ll}&=\int_{\mathbb{T}}\int_{\mathbb{T}} F(\c,\d)\overline{F_{m,n}(\c,\d)}\left[\p,\p\right](\c,\d)\,d\d d\c\\
		&=\int_{\mathbb{T}}\int_{\mathbb{T}} F(\c,\d)\overline{e_{m,n}e^{-\pi imn}}\,d\d d\c\\
		&=\langle F, e_{m,n}e^{-\pi i mn}\rangle_{L^{2}(\mathbb{T}\times \mathbb{T})}\numberthis.
	\end{align*}
    By substituting \eqref{tmn} in \eqref{partialsum}, we obtain
	\begin{align*}\label{smntmn}
		S_{M,N}F&=\sum\limits_{m=-M}^{M}\sum\limits_{n=-N}^{N}\langle F, F_{m,n}\rangle_{L^{2}(\mathbb{T}\times\mathbb{T})} e_{m,n}e^{-2\pi imn}\\
		&=\sum\limits_{m=-M}^{M}\sum\limits_{n=-N}^{N}\langle F, e_{m,n}e^{-\pi i mn}\rangle_{L^{2}(\mathbb{T}\times \mathbb{T})}e_{m,n}e^{-\pi i mn}\\
		&=\sum\limits_{m=-M}^{M}\sum\limits_{n=-N}^{N}\langle F, e_{m,n}\rangle_{\Ll}e_{m,n}\\
		&=T_{M,N}F\numberthis.
	\end{align*}
	By combining \eqref{smnfinite} and \eqref{smntmn}, we get $\sup\limits_{M,N\geq 0}\|T_{M,N}\|<\infty$. Then the result follows by using Lemma \ref{Gaborequivalence}. Conversely, assume that $\left[\p,\p\right]\in \mathcal{A}_{2}(\mathbb{T}\times \mathbb{T}).$ Then there exists $C>0$ such that for all intervals $I,J\subset \mathbb{R},$ we have \[\bigg(\frac{1}{|I||J|}\int_{J}\int_{I}\left[\p,\p\right](x,y)\,dxdy\bigg)\bigg(\frac{1}{|I||J|}\int_{J}\int_{I}\frac{1}{\left[\p,\p\right](x,y)}\,dxdy\bigg)\leq C,\]
	which implies that $\frac{1}{\left[\p,\p\right]}\in L^{1}(\mathbb{T}\times \mathbb{T}).$
	Further, by Lemma \ref{Gaborequivalence}, we have $\sup\limits_{M,N\geq 0}\|T_{M,N}\|<\infty.$ Define $F_{m,n}\in \Ll$ by $F_{m,n}\left[\p,\p\right]=e_{m,n}e^{-\pi imn}.$ Since $\{e_{m,n}e^{-\pi imn} : m,n\in \mathbb{Z}\}$ is an orthonormal system in $L^{2}(\mathbb{T}\times \mathbb{T})$, it can be easily shown that 
	\begin{align*}
		\langle F_{m^{'},n^{'}}, e_{m,n}e^{-\pi i mn}\rangle_{\Ll}=\delta_{(m^{'},n^{'}),(m,n)}.
	\end{align*}
	Therefore $\{F_{m,n} : m,n\in \mathbb{Z}\}$ is biorthogonal to $\{e_{m,n}e^{-\pi imn} : m,n\in \mathbb{Z}\}.$ By the given hypothesis, $\sup\limits_{M,N\geq 0}\|T_{M,N}\|<\infty$. Then it follows from \eqref{smntmn} that $\sup\limits_{M,N\geq 0}\|S_{M,N}\|<\infty$, from which we conclude that $\{\t\p : k,l\in \mathbb{Z}\}$ is a Schauder basis for $V^{t}(\p)$.
\end{proof}

\section{Some comments on dual integrability and Helson map}
Let $H$ denote the locally compact abelian group $2\mathbb{Z}^{n}\times \Z$ and $\mathcal{U}(\l)$ denote the collection of all uniatry operators on $\l$. Now, define a function
\begin{equation*}
	\Pi_{H} : 2\mathbb{Z}^{n}\times \Z\rightarrow \mathcal{U}(\l)~\text{by}~\Pi((2k,l))=T^{t}_{(2k,l)}.
\end{equation*}
Then $\Pi_{H}$ is a unitary representation of $2\mathbb{Z}^{n}\times \Z$ on $\l.$ We can now define Weyl-Zak transform associated with the representation $\Pi_{H}$ as follows. 
\[Z_{\Pi_{H}}(f)(\c,\d,\e)=\sum\limits_{m\in\Z}K_{f}(\c+\frac{m}{2},\e)e^{-2\pi im\cdot \d},\hspace{5mm} \c\in [0,\frac{1}{2})^{n},\d\in \T,\eta\in \mathbb{R}^{n}.\]  Then it can be easily shown that $Z_{\Pi_{H}}$ is an isometric isomorphism from $\l$ onto $L^{2}([0,\frac{1}{2})^{n}\times \T\times \mathbb{R}^{n}).$
\begin{definition}
	The bracket map corresponding to the representation $\Pi_{H}$ is a function \\$\left[\cdot, \cdot\right]_{\Pi_{H}} : \l\times\l \rightarrow L^{1}([0,\frac{1}{2})^{n}\times \T)$, which is defined by 
	\begin{equation*}
		\left[\p,\psi\right]_{\Pi_{H}}(\c,\d)=\int_{\mathbb{R}^{n}}Z_{\Pi_{H}}\p(\c,\d,\e)\overline{Z_{\Pi_{H}}\psi(\c,\d,\e)}\,d\e.
	\end{equation*}
\end{definition}
Let $\p\in \l$. Now, we look at the image of $T^{t}_{(2k,l)}\p$ under the map $Z_{\Pi_{H}}.$
\begin{align*}\label{zak2kl}
	Z_{\Pi_{H}} T^{t}_{(2k,l)}\p(\c,\d,\e)&=\sum\limits_{m\in\Z}K_{T^{t}_{(2k,l)}\p}(\c+\frac{m}{2},\e)e^{-2\pi im\cdot \d}\\
	&=\sum\limits_{m\in\Z}e^{4\pi i(\c+\frac{m}{2})\cdot k}K_{\p}(\c+\frac{m}{2}+l,\e)e^{-2\pi im\cdot \d}\\
	&=e^{4\pi i\c\cdot k}\sum\limits_{m\in\Z}K_{\p}(\c+\frac{m}{2},\e)e^{-2\pi i(m-2l)\cdot \d}\\
	&=e^{2\pi i\c\cdot 2k}(e^{2\pi il\cdot \d})^{2}Z_{\Pi_{H}}\p(\c,\d,\e)\numberthis.
\end{align*} 
For any $k,l\in \Z,$ 
\begin{align*}\label{bracketpi}
	\langle\p,\Pi_{H}(2k,l)\psi\rangle_{\l}&=\langle\p,T^{t}_{(2k,l)}\psi\rangle_{\l}\\
	&=\langle Z_{\Pi_{H}}\p,Z_{\Pi_{H}} T^{t}_{(2k,l)}\psi\rangle_{L^{2}([0,\frac{1}{2})^{n}\times \T\times \mathbb{R}^{n})}\\
	&=\int_{[0,\frac{1}{2})^{n}}\int_{\mathbb{T}^{n}}e^{-2\pi i\c\cdot 2k}e^{-4\pi il\cdot \d}\bigg(\int_{\mathbb{R}^{n}}Z_{\Pi_{H}}\p(\c,\d,\e)\overline{Z_{\Pi_{H}}\psi(\c,\d,\e)}d\e\bigg)\,d\d\,d\c\\
	&=\int_{[0,\frac{1}{2})^{n}}\int_{\mathbb{T}^{n}}\left[\p,\psi\right]_{\Pi}(\c,\d)e^{-2\pi i\c\cdot 2k}e^{-2\pi il\cdot \d}e^{-2\pi il\cdot \d}\,d\d\,d\c,\numberthis
\end{align*}
by using \eqref{zak2kl}.\\

Define a map $\mathcal{J} : \l\rightarrow L^{2}(\mathbb{R}^{n},L^{2}([0,\frac{1}{2})^{n}\times \T))$ by \[\mathcal{J}f(\e)(\c,\d)=Z_{\Pi_{H}} f(\c,\d,\e).\] Then \eqref{zak2kl} can be written as 
\begin{equation}\label{helsonmap}
	\mathcal{J}[\Pi_{H}(2k,l)f](\e)(\c,\d)=e^{2\pi i\c\cdot 2k}e^{2\pi il\cdot \d}e^{2\pi il\cdot \d}\mathcal{J}[f](\e)(\c,\e).
\end{equation}

Observe that there is a presence of an additional term in the integrand in \eqref{bracketpi} and in \eqref{helsonmap} in comparison to \eqref{dualdef1} and \eqref{helson1} respectively. We shall try to explain the reason behind this. If we consider $G=\mathbb{Z}$, integer translations (unitary representation of $\mathbb{Z}$ on $\mathcal{B}(L^{2}(\mathbb{R}))$) and the representation $\Pi : \mathbb{Z}\mapsto \mathcal{B}(L^{2}(\mathbb{R}))$ defined by $\Pi(k)=T_{k}$ along with the Helson map $\mathcal{J} : L^2(\mathbb{R})\mapsto \ell^{2}(\mathbb{Z},L^{2}(\mathbb{T}))$ defined by $\mathcal{J} f(m)(\xi)=\widehat{f}(\c+m)$, $f\in L^{2}(\mathbb{R}),\c\in \mathbb{T},m\in \mathbb{Z}$, then $$\langle \phi, \Pi(k)\psi \rangle_{L^{2}(\mathbb{R})}=\int_{\mathbb{T}}\bigg(\sum\limits_{m\in \mathbb{Z}}\widehat{\phi}(\c+m)\overline{\widehat{\psi}(\c+m)}\bigg)e^{-2\pi ik\c}\,d\c~\text{and}~\mathcal{J} (\Pi(k)\phi)(m)(\c)=e^{2\pi ik\c}\widehat{\phi}(\c+m).$$ Thus $\Pi$ satisfies (1.1) and (1.2). This is due to the fact that $\widehat{T_{k}f}(\c)=e^{-2\pi ik\c}\widehat{f}(\c)$. However, in the present setting of twisted translates, we have \begin{equation}\label{final}
	K_{\t\phi}(\c,\e)=e^{\pi i(2\c+l)k}K_{\phi}(\c+l,\e),~\text{for}~\phi\in L^{2}(\mathbb{R}^{2n}),
\end{equation} by Lemma \ref{twistker}. We notice that, on the R.H.S $``\c+l"$ appears as the first variable of $K_{\phi}$ instead of simply $``\c"$. This leads to the additional terms in (7.2) and (7.3) in comparison to (1.1) and (1.2). Recall that the function $K_{\phi}$ is the kernel of the Weyl transform of $\phi$ as mentioned in Section 2. Thus \eqref{final} is equivalent to $W(\t \phi)\chi(\c)=e^{\pi i(2\c+l)k}W(\phi)\chi(\c+l)$, for $\chi\in L^{2}(\mathbb{R}^{n})$. Further, the Weyl transform is obtained from $W_{\lambda}(f^{\lambda})=\widehat{f}(\lambda)$, (by taking $\lambda=1$), which is the group Fourier transform of $f$ on $\mathbb{H}^{n}$ that satisfies $\widehat{L_{(2k,l,m)}f}(\lambda)=\pi_{\lambda}(2k,l,m)\widehat{f}(\lambda)$, where $L_{(2k,l,m)}f$ denotes the left translate of $f$ on the Heisenberg group $\mathbb{H}^{n}$ by $(2k,l,m)$. Thus, the additional terms in (7.1) and (7.2) are due to the property of the Schr\"{o}dinger representation on the non-abelian group $\mathbb{H}^{n}$. We believe that these observations pave a way for further investigations in future towards Helson map and dual integrability on a non-abelian locally compact group.  

\section*{Acknowledgement}
We thank Prof. C. Heil for his helpful comments on Lemma \ref{Gaborequivalence}. We thank the referee for meticulously reading the manuscript and giving us valuable comments and suggestions which helped us to improve the presentation of the earlier version of the manuscript to the current version.\\

\bibliographystyle{amsplain}
\bibliography{weylzak}

\providecommand{\bysame}{\leavevmode\hbox to3em{\hrulefill}\thinspace}
\providecommand{\MR}{\relax\ifhmode\unskip\space\fi MR }
\providecommand{\MRhref}[2]{%
  \href{http://www.ams.org/mathscinet-getitem?mr=#1}{#2}
}
\providecommand{\href}[2]{#2}
\begin{thebibliography}{10}

\bibitem{Arati}
S.~Arati and R.~Radha, \emph{Frames and {R}iesz bases for shift invariant
  spaces on the abstract {H}eisenberg group}, Indag. Math. (N.S.) \textbf{30}
  (2019), no.~1, 106--127. \MR{3906124}

\bibitem{Mayeli}
D.~Barbieri, E.~Hern\'{a}ndez, and A.~Mayeli, \emph{Bracket map for the
  {H}eisenberg group and the characterization of cyclic subspaces}, Appl.
  Comput. Harmon. Anal. \textbf{37} (2014), no.~2, 218--234. \MR{3223463}

\bibitem{Barbieri_ACHA2015}
D.~Barbieri, E.~Hern\'{a}ndez, and J.~Parcet, \emph{Riesz and frame systems
  generated by unitary actions of discrete groups}, Appl. Comput. Harmon. Anal.
  \textbf{39} (2015), no.~3, 369--399. \MR{3398942}

\bibitem{barbieri}
D.~Barbieri, E.~Hern\'{a}ndez, and V.~Paternostro, \emph{Spaces invariant under
  unitary representations of discrete groups}, J. Math. Anal. Appl.
  \textbf{492} (2020), no.~1, 124357, 32. \MR{4126765}

\bibitem{bownik}
M.~Bownik, \emph{The structure of shift-invariant subspaces of {$L^2({\R})$}},
  J. Funct. Anal. \textbf{177} (2000), no.~2, 282--309. \MR{1795633}

\bibitem{bownikr}
M.~Bownik and K.~A. Ross, \emph{The structure of translation-invariant spaces
  on locally compact abelian groups}, J. Fourier Anal. Appl. \textbf{21}
  (2015), no.~4, 849--884. \MR{3370013}

\bibitem{cabrelli}
C.~Cabrelli and V.~Paternostro, \emph{Shift-invariant spaces on {LCA} groups},
  J. Funct. Anal. \textbf{258} (2010), no.~6, 2034--2059. \MR{2578463}

\bibitem{CO}
O.~Christensen, \emph{Frames and bases}, Applied and Numerical Harmonic
  Analysis, Birkh\"{a}user Boston, Inc., Boston, MA, 2008, An introductory
  course. \MR{2428338}

\bibitem{CN}
O.~Christensen, \emph{An introduction to frames and {R}iesz bases}, second ed.,
  Applied and Numerical Harmonic Analysis, Birkh\"{a}user/Springer, [Cham],
  2016. \MR{3495345}

\bibitem{currey}
B.~Currey, A.~Mayeli, and V.~Oussa, \emph{Characterization of shift-invariant
  spaces on a class of nilpotent {L}ie groups with applications}, J. Fourier
  Anal. Appl. \textbf{20} (2014), no.~2, 384--400. \MR{3200927}

\bibitem{Twisted}
S.~R. Das, P.~Massopust, and R.~Radha, \emph{Twisted {B}-splines in the complex
  plane}, Appl. Comput. Harmon. Anal. \textbf{56} (2022), 250--282.
  \MR{4310941}

\bibitem{das}
S.~R. Das and R.~Radha, \emph{Shift-invariant system on the {H}eisenberg
  {G}roup}, Adv. Oper. Theory \textbf{6} (2021), no.~1, Paper No. 21, 27.
  \MR{4181682}

\bibitem{follandphase}
G.~B. Folland, \emph{Harmonic analysis in phase space}, Annals of Mathematics
  Studies, vol. 122, Princeton University Press, Princeton, NJ, 1989.
  \MR{983366}

\bibitem{Heilbook}
C.~Heil, \emph{A basis theory primer}, expanded ed., Applied and Numerical
  Harmonic Analysis, Birkh\"{a}user/Springer, New York, 2011. \MR{2744776}

\bibitem{Heil}
C.~Heil and A.~M. Powell, \emph{Gabor {S}chauder bases and the {B}alian-{L}ow
  theorem}, J. Math. Phys. \textbf{47} (2006), no.~11, 113506, 21. \MR{2278667}

\bibitem{tribute}
E.~Hern\'{a}ndez, P.~M. Luthy, H.~\v{S}iki\'{c}, F.~Soria, and E.~N. Wilson,
  \emph{Spaces generated by orbits of unitary representations: a tribute to
  {G}uido {W}eiss}, J. Geom. Anal. \textbf{31} (2021), no.~9, 8735--8761.
  \MR{4302196}

\bibitem{cyclic}
E.~Hern\'{a}ndez, H.~\v{S}iki\'{c}, G.~Weiss, and E.~Wilson, \emph{Cyclic
  subspaces for unitary representations of {LCA} groups; generalized {Z}ak
  transform}, Colloq. Math. \textbf{118} (2010), no.~1, 313--332. \MR{2600532}

\bibitem{biorthogonal}
E.~Hern\'{a}ndez, H.~\v{S}iki\'{c}, G~Weiss, and E~Wilson, \emph{On the
  properties of the integer translates of a square integrable function},
  Harmonic analysis and partial differential equations, Contemp. Math., vol.
  505, Amer. Math. Soc., Providence, RI, 2010, pp.~233--249. \MR{2664571}

\bibitem{wheeden}
R~Hunt, B~Muckenhoupt, and R~Wheeden, \emph{Weighted norm inequalities for the
  conjugate function and {H}ilbert transform}, Trans. Amer. Math. Soc.
  \textbf{176} (1973), 227--251. \MR{312139}

\bibitem{Iverson}
J.~W. Iverson, \emph{Frames generated by compact group actions}, Trans. Amer.
  Math. Soc. \textbf{370} (2018), no.~1, 509--551. \MR{3717988}

\bibitem{kamyabi}
R.~A. Kamyabi~Gol and R.~R. Tousi, \emph{The structure of shift invariant
  spaces on a locally compact abelian group}, J. Math. Anal. Appl. \textbf{340}
  (2008), no.~1, 219--225. \MR{2376149}

\bibitem{Mallat1989}
S.~G. Mallat, \emph{Multiresolution approximations and wavelet orthonormal
  bases of {$L^2({\mathbb{R}})$}}, Trans. Amer. Math. Soc. \textbf{315} (1989),
  no.~1, 69--87. \MR{1008470}

\bibitem{Meyer1987}
Y.~Meyer, \emph{Ondelettes et fonctions splines}, S\'{e}minaire sur les
  \'{e}quations aux d\'{e}riv\'{e}es partielles 1986--1987, \'{E}cole
  Polytech., Palaiseau, 1987, pp.~Exp. No. VI, 18. \MR{920024}

\bibitem{Sikic_2007}
M.~Nielsen and H.~\v{S}iki\'{c}, \emph{Schauder bases of integer translates},
  Appl. Comput. Harmon. Anal. \textbf{23} (2007), no.~2, 259--262. \MR{2344615}

\bibitem{saswata}
R.~Radha and S.~Adhikari, \emph{Frames and {R}iesz bases of twisted
  shift-invariant spaces in {$L^2(\Bbb{R}^{2n})$}}, J. Math. Anal. Appl.
  \textbf{434} (2016), no.~2, 1442--1461. \MR{3415732}

\bibitem{RHG}
R.~Radha and S.~Adhikari, \emph{Shift-invariant spaces with countably many
  mutually orthogonal generators on the {H}eisenberg group}, Houston J. Math.
  \textbf{46} (2020), no.~2, 435--463. \MR{4195268}

\bibitem{radhas}
R.~Radha and N.~S. Kumar, \emph{Shift invariant spaces on compact groups},
  Bull. Sci. Math. \textbf{137} (2013), no.~4, 485--497. \MR{3054272}

\bibitem{thangavelu}
S.~Thangavelu, \emph{Harmonic analysis on the {H}eisenberg group}, Progress in
  Mathematics, vol. 159, Birkh\"{a}user Boston, Inc., Boston, MA, 1998.
  \MR{1633042}

\end{thebibliography}
\end{document}